\documentclass[11pt]{amsart}

\usepackage{amsbsy, amsmath,amsfonts,amssymb,amsthm,amscd,amssymb,latexsym,}
\usepackage[latin1]{inputenc}
\usepackage[dvips]{epsfig}
\usepackage{graphics}
\usepackage{psfrag}
 \usepackage{amsfonts,amsmath,amsthm}
\usepackage{amssymb}
\usepackage[all]{xy}
\usepackage{latexsym}
\usepackage{enumerate}
\usepackage{graphicx,color}
\usepackage{epsfig}
\usepackage{graphicx}

\usepackage{indentfirst}

\usepackage{amsbsy, amsmath,amsfonts,amssymb,amsthm,amscd,amssymb,latexsym}
\usepackage{graphicx}
\usepackage[dvips]{epsfig}
\usepackage{graphics}
\usepackage{psfrag}

\usepackage{cancel}

\usepackage{enumerate}
\usepackage{indentfirst}
\usepackage{euscript}

 \theoremstyle{plain}

\newcommand{\Cc}{\mathcal C}
\newcommand{\Pp}{\mathcal P}
\newcommand{\Qq}{\mathcal Q}
\newcommand{\Gg}{\mathcal G}
\newcommand{\Ii}{\mathcal I}

\newcommand{\Vv}{\mathcal V}
\newcommand{\Uu}{\mathcal U}
\newcommand{\Xx}{\mathcal X}
\newcommand{\Yy}{\mathcal Y}
\newcommand{\Ff}{\mathcal F}

\newcommand{\p}{\mathfrak p}

\newcommand{\R}{\mathbb R}

\newcommand{\LL}{\mathbb L}
\newcommand{\Hess}{\mbox{Hess}}

\newcommand{\cv}{\cdot}

\newcommand{\nin}{\noindent}
\newcommand{\fun}{\longrightarrow}

\newcommand{\ep}{\varepsilon}

\newcommand{\lp}{\langle}
\newcommand{\rp}{\rangle}

\newtheorem{Teo}{Theorem} 

\theoremstyle{definition}
\newtheorem{defn}[Teo]{Definition}
\newtheorem{prop}[Teo]{Proposition}
\newtheorem{lema}[Teo]{Lemma}

\theoremstyle{remark}
\newtheorem{obs}[Teo]{Remark}
\newtheorem{af}[Teo]{Claim}

{\noindent{\bfseries Proof: }}
{\hfill$\rule{1.2ex}{1.2ex}$} 

\newenvironment{prova}
{\noindent \texttt{Proof:}}
{\hfill$\Diamond$} 

\usepackage[usenames,dvipsnames]{xcolor}
 \usepackage{color}

 \title[ Axiumbilic Singular Points on  Surfaces Immersed  in  $\R^4 $]{ Axiumbilic Singular Points on  Surfaces Immersed   in  $\R^4$      and their  Generic Bifurcations}

 \author{   R. Garcia, J. Sotomayor and F. Spindola}

\begin{document}

\begin{abstract}
Here are described the {\it  axiumbilic}
points
 that appear in
generic
one parameter families
of surfaces immersed in  $\R^4$. At these points
 the ellipse of curvature of the immersion,   Little
\cite{little},
Garcia - Sotomayor  \cite{sotogarcia1},  has equal axes.

A review   is made on the basic preliminaries on axial curvature lines and the associated
axiumbilic points which are the singularities of the   fields of
{ \it principal},
{ \it mean axial lines},
{ \it axial crossings} and
the
quartic differential equation
defining them.

   The  Lie-Cartan  vector field suspension of the quartic differential equation,  giving a line field tangent to the
  Lie-Cartan surface  (in the projective bundle of the source immersed surface
  which quadruply covers
  a
  punctured
   neighborhood of the axiumbilic point) whose integral curves project regularly on the lines of axial curvature.

  In an appropriate  Monge chart   the configurations of the generic  axiumbilic  points, denoted by  $E_3$, $E_4$ and $E_5$  in \cite{sotogarcia1} \cite{sotogarcia},  are obtained by studying the integral curves of the Lie-Cartan  vector field.

Elementary bifurcation theory is applied to the study of  the transition and elimination  between  the axiumbilic  generic points.
The two generic patterns  $E^1_{34}$ and $E^1_{45}$ are analysed and their axial configurations are  explained
in terms of
 their qualitative changes  (bifurcations)
 with one  parameter in the space of  immersions, focusing on their  close analogy
 with
  the saddle-node
 bifurcation for  vector fields in the plane  \cite{andronov}, \cite{soto1}.

This work can be regarded as a partial  extension  to  $\mathbb R^4$
of  the umbilic bifurcations in Garcia - Gutierrez - Sotomayor
\cite{soto-ga-gu}, for surfaces in $\mathbb R^3$.
With less restrictive differentiability hypotheses and distinct methodology
it has points of contact with  the results of Gutierrez - Gui\~nez - Casta\~neda \cite{gutierrez3}.

\end{abstract}
\maketitle
 
\section*{Introduction}

In this work  are described the axiumbilic singularities, at which the ellipse of curvature,
as defined in Little \cite{little} and  Garcia - Sotomayor
\cite{sotogarcia1},  has equal axes.
 The focus here  are the axiumbilic points  that appear
 generically
 in one parameter families
of surfaces immersed in  $\R^4$.
 It can be regarded as an  extension from $\R^3$ to $\R^4$,  as target spaces for immersed  surfaces, and from umbilic to axiumbilic points as singularities,
 of results obtained  by Gutierrez - Garcia - Sotomayor in \cite{soto-ga-gu}.
 It is also a continuation, in the direction   of bifurcations of  axiumbilic singularities, of the study  of the structural stability of
   global axial configurations  started in Garcia - Sotomayor
\cite{sotogarcia1}.

An outline of the  organization of this paper  follows:

Section \ref{sec:1}  deals  with geometric preliminaries
and  a review of
axial lines and axiumblic points in order to define the
{\it principal}
 and
 {\it mean curvature}
configurations and their quartic differential equations.

In Section \ref{sec:2},  locally  presenting  a surface  $M$ immersed into $\R^4$ with a Monge chart, are  studied the axiumbilic points  and the
transversality conditions in terms of which are defined the generic axiumbilic points are made explicit.

 Section \ref{sec:3} establishes   the axial  principal and mean configurations in a neighborhood  of generic axiumbilic points, denoted   $E_3$, $E_4$
   and $E_5$.
    This description uses the suspension of Lie-Cartan,
     giving rise to a line field tangent to a surface, which quadruply covers
  a
  punctured
  neighborhood of the axiumbilic point, and whose integral lines project regularly on the lines of axial curvature.
    This follows   the approach of  Garcia and  Sotomayor in \cite{sotogarcia1}  and \cite{sotogarcia}, chap. 8.

After this review follow two  subsection  devoted to describe the behaviors
 of axial lines near the axiumbilic  points
denoted $E^1_{34}$ and  $E^1_{45}$,  which are the
transversal
 transitions between  the generic axiumbilic points.
 
 In fact, the axiumbilic  point  $E^1_{34}$ (Figure \ref{e34}) characterizes the  transition between an axiumbilic point of type $E_3$ and one of type  $E_4$, which is explained  by the  variation of one parameter family in the  space
  of immersions  ${\Cc}^r,\ r\geq 5$ of a surface  $M$ into  $\R^4$ (Proposition \ref{proposicao.e34}), in a first analogy with the saddle-node  bifurcation
  of vector fields  \cite{andronov},  \cite{soto1}.

  The axiumbilic point
 $E^1_{45}$ (Figure \ref{desenho.e45}) is characterized by
 the
 collision and subsequent elimination  between  one  point
 of type
 $E_4$ and  other of type   $E_5$.
 Here also, this  bifurcation phenomenon is explained by means of a one parameter variation in the space of immersions  (Proposition \ref{proposicao.e45}), in a second analogy  with the  saddle-node bifurcations in the  plane \cite{andronov}  \cite{soto1}.

Section \ref{sec:4}  establishes the genericity of the axiumbilic bifurcations studied in  this paper.

 This work  can be  related to   the papers by  Gu\'i\~nez-Guti\'errez \cite{gutierrez1} and Gu\'i\~nez-Guti\'errez-Casta\~neda \cite{gutierrez3}
  where  a description, in  class  ${\Cc}^{\infty}$ and  in the context of quartic differential forms, of the  points
   $E^1_{34}$ and $E^1_{45}$ (using the notation $H_{34}$ and $H_{45}$), can be found.

Here was adopted a different approach, using the Lie-Cartan suspension as established in   Garcia-Sotomayor \cite{sotogarcia1},
for immersions of class
${\Cc}^r,  5\leq  r \leq \infty $.
This  leads to an interpretation of these points with less  restrictive differentiability hypotheses and allows proofs with
 techniques   closer to those of
elementary bifurcation theory as in \cite {andronov} and \cite{soto1}.

Section \ref{sec:5} closes the paper with related comments on its  results and their connection with others found  in the literature.

\section{Differential Equation of Axial Lines }\label{sec:1}

Let  $\alpha: M  \fun \R^4$ be an immersion of class ${\Cc}^r$, $r\geq 5$, of an oriented smooth surface in  $\R^4$, with the canonical orientation.
Assume that $(x,y)$ is a positive chart of $M $ and that  $\{\alpha_x,\alpha_y,N_1,N_2\}$ is a smooth
positive  frame in $\R^4$,
where  for
$\p \in M$, $\{\alpha_x =\partial{\alpha}/\partial{x}, \alpha_y = \partial{\alpha}/\partial{y}\}_{\p}$  is  the
the standard
  basis of $T_{\p}M$  in the chart
$(x,y)$
   and  $\{N_1,N_2\}_{\p}$ is a basis of the normal plane  $N_{\p} M$.

In the chart  $(x,y)$, the first fundamental form is expressed by
$$I_{\alpha} = \langle D\alpha, D\alpha\rangle = E dx^2 + 2Fdxdy + G dy^2$$
 where,
$E=\lp \alpha_x,\alpha_x \rp, F=\lp\alpha_x,\alpha_y\rp$ and $G=\lp\alpha_y,\alpha_y\rp$
and the second  fundamental form is given by
$II_{\alpha} = II^1_{\alpha} N_1 + II^2_{\alpha} N_2$
where  $II^i_{\alpha}, i=1,2,$  is
$$II^i_{\alpha} := \langle D^2 \alpha, N_i\rangle = e_i dx^2+2f_i dx dy+g_i dy^2$$
being
$e_i=\lp\alpha_{xx},N_i\rp, f_i=\lp\alpha_{xy},N_i\rp$ and  $g_i=\lp\alpha_{yy},N_i\rp$.

The  \textit{mean curvature vector } is defined by $H=h_1 N_1 +h_2 N_2$
with
$$h_i=\frac{Eg_i-2Ff_i+Ge_i}{2(EG-F^2)}.$$

For
  $v \in T_{\p}M$, the  \textit{normal curvature vector in the direction    $v$}  is defined by:
\begin{equation}
\label{kn}
k_n=k_n(\p,v) = \frac{II_{\alpha}(v)}{I_{\alpha}(v)}= \frac{II^1_{\alpha}(v)}{I_{\alpha}(v)}N_1 + \frac{II^2_{\alpha}(v)}{I_{\alpha}(v)}N_2 .
\end{equation}

 The image of  $k_n$ restricted to the unitary circle   $S^1_{\p}$ of $T_{\p}M$ describes  in $N_{\p}M$ an ellipse centered in  $H(\p)$, which is called
  \textit{ellipse of curvature}  of $\alpha$ at $\p$, and it will be denoted by $\ep_{\alpha}(\p)$.
	
	\nin When  $(e_1-g_1)f_2 - (e_2-g_2)f_1\neq 0$, it is an actual non-degenerate ellipse, which can be a circle. Otherwise it can be a segment or a point.
 As  $k_n  |_{S^1_{\p}}$ is quadratic, the pre-image of each point of the ellipse is formed of two antipodal points on   $S^1_{\p}$, and therefore   each point $\ep_{\alpha}(\p)$ is associated to a direction  in  $T_{\p}M$.
 Moreover, for each pair
 of points in   $\ep_{\alpha}(\p)$
antipodally
 symmetric with respect
  to  $H(\p)$, it is associated two orthogonal directions in  $T_{\p}M$, defining a pair of   \textit{lines} in  $T_pM$  \cite{little}, \cite{luisfernando}, \cite{lf}.

Consider the function:

\begin{eqnarray*}
\|k_n - H\|^2
&:=&
\bigg [ \frac{e_1 dx^2+2f_1 dx dy+g_1 dy^2}{E dx^2 + 2F dx dy + G dy^2} - \frac{Eg_1-2Ff_1+Ge_1}{2(EG-F^2)} \bigg ]^2 \\
&+&
\bigg [ \frac{e_2 dx^2+2f_2 dx dy+g_2 dy^2}{E dx^2 + 2F dx dy + G dy^2} - \frac{Eg_2-2Ff_2+Ge_2}{2(EG-F^2)}  \bigg ]^2
\end{eqnarray*}

For each  $\p \in M$ in which  $\ep_{\alpha}(\p)$ is not a circle,
the  points maximum  and minimum  of this function determine  four  points
over the  ellipse of  curvature $\ep_{\alpha}(\p)$, which are their vertices,
located at the
large
 and
small
 axes.

\begin{figure}[ht]
       \centering  
       \includegraphics[height=5cm, width=11cm]{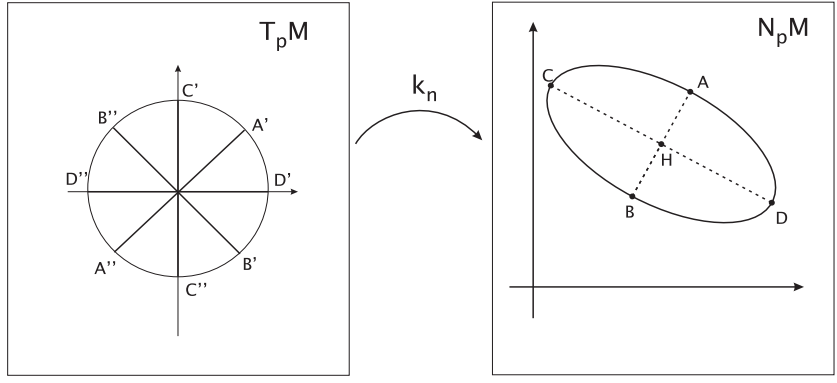}
\caption{Ellipse of Curvature $\ep_{\alpha}(\p)$ and Lines of Axial Curvature}
\label{elipse2}
\end{figure}

As illustrated in  Figure  \ref{elipse2}, to the small axis   $AB$
is associated the crossing  $A'A''B'B''$ and to the large axis  $CD$ is associated
the crossing
 $C'C''D'D''$.
 Thus, for each  $\p \in M$ at which the non-degenerate ellipse is not a circle or a point,
\textit{ two crossings}  are defined in  $T_{\p}M$, one  associated to
the large axis and the other to the small axis
of the  ellipse of  curvature.
These  \textit{Fields  of 2-Crossings} in  $M$ are called  \textit{Fields of  Axial  Curvature}.

Outside the set
 ${\Uu}_{\alpha}$ of points at which the ellipse of
curvature is a circle (i.e. has equal axes), called
\textit{Axiumbilic Points},
the lines and crossings are
said to be
\textit{Lines and Crossings of  Axial Curvature}. Those related to the large  (respectively small)  axis of
the ellipse  of curvature are called  \textit{Lines and Crossings of  Principal (respectively  Mean)  Axial Curvature}.

From the considerations above, the axial directions are defined by the equation

$$Jac(\|k_n-H\|^2, I_{\alpha})=0$$
which has four solutions for  $\p\notin  {\Uu}_{\alpha}$ and is singular at  $\p \in {\Uu}_{\alpha}$. According to  \cite{sotogarcia1} and  \cite{sotogarcia}, the  \textit{differential equation of axial lines}  is given by:
\begin{equation}
\label{quartica.axiais}
a_4 dy^4+a_3 dy^3 dx + a_2 dy^2 dx^2 + a_1 dy dx^3 + a_0 dx^4 = 0,
\end{equation}
where
\begin{eqnarray*}
a_4
&=&
-4F(EG-2F^2)(g_1^2+g_2^2)+4G(EG-4F^2)(f_1g_1+f_2g_2), \\
&+&
8FG^2(f_1^2+f_2^2)+4FG^2(e_1g_1+e_2g_2)-4G^3(e_1f_1+e_2f_2)
\\
\\
a_3
&=&
-4E(EG-4F^2)(g_1^2+g_2^2)-32EFG(f_1g_1+f_2g_2), \\
&+&
16EG^2(f_1^2+f_2^2)-4G^3(e_1^2+e_2^2)+8EG^2(e_1g_1+e_2g_2)
\\
\\
a_2
&=&
-12FG^2(e_1^2+e_2^2)+12E^2F(EG-4F^2)(g_1^2+g_2^2), \\
&+&
24EG^2(e_1f_1+e_2f_2)-24E^2G(f_1g_1+f_2g_2)
\\
\\
a_1
&=&
4E^3(g_1^2+g_2^2)+4G(EG-4F^2)(e_1^2+e_2^2) \\
&+&
32EFG(e_1f_1+e_2f_2)-16E^2G(f_1^2+f_2^2)-8E^2G(e_1g_1+e_2g_2),
\\
\\
a_0
&=&
4F(EG-2F^2)(e_1^2+e_2^2)-4E(EG-4F^2)(e_1f_1+e_2f_2) \\
&+&
-8E^2F(f_1^2+f_2^2)-4E^2F(e_1g_1+e_2g_2)+4E^3(f_1g_1+f_2g_2).
\end{eqnarray*}

\begin{prop}[\cite{sotogarcia1}, \cite{sotogarcia}]
\label{eq.dif.axial}
Let  $\alpha: M \fun \R^4$ be an immersion of class  ${\Cc}^r, \ r\geq 5$, of an oriented
 and smooth surface. Denote the first fundamental form of $\alpha$    by
$$I_{\alpha} = E dx^2 + 2Fdxdy + G dy^2$$
and the second fundamental form by:
$$II_{\alpha}= (e_1 dx^2+2f_1 dx dy+g_1 dy^2)N_1 + (e_2 dx^2+2f_2 dx dy+g_2 dy^2) N_2$$
where $\{N_1,N_2\}$  is an orthonormal frame.  
 
\begin{enumerate}[$i)$]
	\item  The differential equation of axial lines is given by:
\begin{eqnarray*}
    {\Gg}
&=& [a_0 G(EG-4F^2)+a_1 F(2F^2-EG)]dy^4 \\
&+& [-8a_0 EFG + a_1 E(4F^2-EG)]dy^3 dx \\
&+& [-6a_0 GE^2 + 3a_1 FE^2]dy^2 dx^2 +
  a_1 E^3 dy dx^3
+  a_0 E^3 dx^4 =0,
\end{eqnarray*}
where
\begin{eqnarray*}
		a_1
&=& 4G(EG-4F^2)(e_1^2 + e_2^2)
+ 32 EFG(e_1f_1+e_2f_2) \\
&+& 4E^3(g_1^2 + g_2^2)
- 8E^2G (e_1g_1+e_2g_2)
-  16E^2G (f_1^2 + f_2^2)
\end{eqnarray*}
and
\begin{eqnarray*}
		a_0
&=& 4F(EG-2F^2)(e_1^2 + e_2^2)
-  4E(EG-4F^2)(e_1f_1+e_2f_2) \\
&+& 4E^3(f_1g_1 + f_2g_2)
-  4E^2F (e_1g_1+e_2g_2) -
 8E^2F (f_1^2 + f_2^2).
\end{eqnarray*}

\item The axiumbilic points of  $\alpha$ are characterized by  $a_0=a_1=0$.
\end{enumerate}

\end{prop}
 The axiumbilic points are defined by the intersection of the curves
 $a_0(x,y)=0$ and $a_1(x,y)=0$.  Assume, with no lost
  of generality, that they intersect at   $(x,y)=(0,0)$. In this work it will be considered the case where the intersection
 is transversal or quadratic at  $(0,0)$.

Figure  \ref{transversalidade} illustrates the generic contact of the curves $a_0(x,y)=0$ and $a_1(x,y)=0$, whose intersection characterizes the axiumbilic points.

\begin{figure}[h]
       \centering  
       \fbox{\includegraphics[scale=0.39]{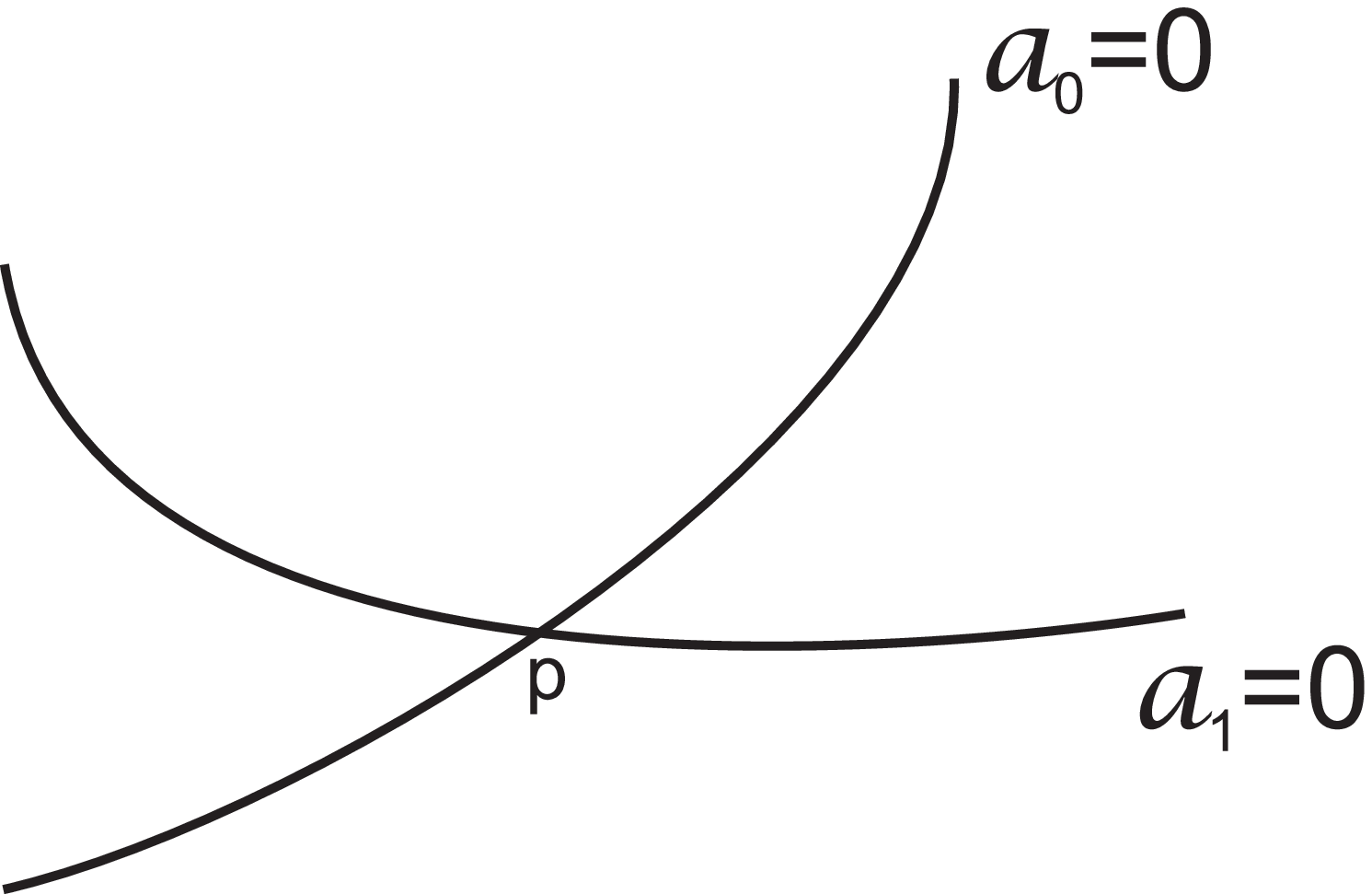}}
       \fbox{\includegraphics[scale=0.39]{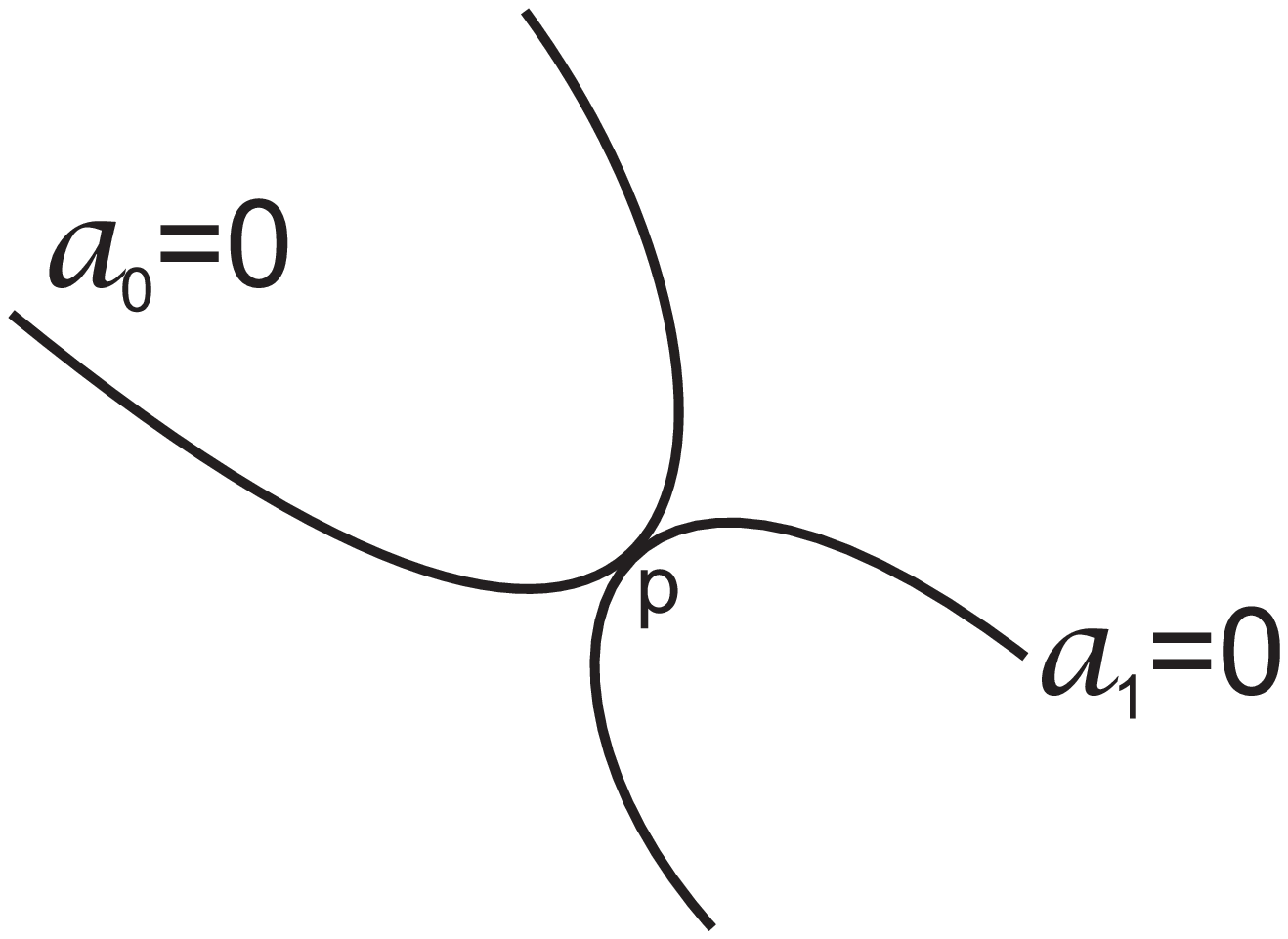}}
\caption{Transversal and quadratic contact between the curves  $a_0=0$  and $a_1=0$ at an axiumbilic point $\p$.}
\label{transversalidade}
\end{figure}

An axiumbilic point given by   $(x,y)=(0,0)$  is called \emph{transversal}  if
\begin{equation}
\label{determinante.transversal}
\frac{\partial (a_0,a_1)}{\partial(x,y)}\bigg |_{(0,0)} =
\left |
\begin{array}{cc}
	\frac{\partial a_0}{\partial x}(0,0) & \frac{\partial a_0}{\partial y}(0,0) \\
	\frac{\partial a_1}{\partial x}(0,0) & \frac{\partial a_1}{\partial y}(0,0)
\end{array}
\right |
\neq  0.
\end{equation}

The axiumbilic point given by
  $(x,y)=(0,0)$ is said to be  of \emph{  quadratic type} if  the matrix
\begin{equation}
\label{determinante.quadratico}
\frac{\partial (a_0,a_1)}{\partial(x,y)}\bigg |_{(0,0)} =
\left [
\begin{array}{cc}
	\frac{\partial a_0}{\partial x}(0,0) & \frac{\partial a_0}{\partial y}(0,0) \\
	\frac{\partial a_1}{\partial x}(0,0) & \frac{\partial a_1}{\partial y}(0,0)
\end{array}
\right ]
\end{equation}
has rank  $1$ and, assuming $\frac{\partial a_0}{\partial y}(0,0) \neq 0$,
it follows from the Implicit Function Theorem that     $y(x)$ is a local solution of $a_0(x,y(x))=0$. Writing  $s(x)=a_1(x,y(x))$ it follows that $s'(0)=0$ and $s''(0)\neq 0$.

A similar analysis can be carried out if other element of the matrix
  $\frac{\partial (a_0,a_1)}{\partial(x,y)}\bigg |_{(0,0)}$ is non zero.

\begin{obs}[\cite{sotogarcia1}]
In isothermic coordinates, where $E=G$ and $F=0$, it follows that
$$a_1=-a_3=E^3[e_1^2+e_2^2+g_1^2+g_2^2-4(f_1^2+f_2^2)-2(e_1g_1+e_2g_2)]$$
$$a_0=a_4=-\frac{a_2}{6}=4E^3[f_1g_1+f_2g_2-(e_1f_1+e_2f_2)]$$
and the differential equation of axial lines is  simplified to

\begin{equation}
\label{axiais}
a_0(x,y)(dx^4-6dx^2dy^2+dy^4)+a_1(x,y)(dx^2-dy^2)dx dy=0.
\end{equation}

\end{obs}

\subsection{Axial Configurations of immersed surfaces in $\R^4$}

Let ${\Ii}^r = {\Ii}^r(M,\R^4)$ the set of immersions of class ${\Cc}^r$. For $\alpha \in {\Ii}^r$, the differential equation of axial lines is well defined (equation (\ref{quartica.axiais})):
\begin{equation}
\label{quartica.axiais.2}
{\Gg}(x,y,dx,dy) = a_4 dy^4+a_3 dy^3 dx + a_2 dy^2 dx^2 + a_1 dy dx^3 + a_0 dx^4 = 0
\end{equation}
in the projective bundle  $PM$ of  $M$.

For each $\alpha \in {\Ii}^r$, define the \textit{  Lie-Cartan surface } of the immersion $\alpha$ by ${\LL}_{\alpha}:={\Gg}_{\alpha}^{-1}(0)$, which is of class  ${\Cc}^{r-2}$, regular in      $M-{\Uu}_{\alpha}$ and may present singularities at ${\Uu}_{\alpha}$. Moreover, as the differential equation  (\ref{quartica.axiais.2})  is quartic and contains the projective line at    ${\Uu}_{\alpha}$, it follows that   ${\LL}_{\alpha}$ is a ramified covering of degree  $4$ in  $M-{\Uu}_{\alpha}$ and  contains the projective line $\pi^{-1}(\p)$ for each $\p \in {\Uu}_{\alpha}$.

In the chart $(x,y,p)$, with $p=\frac{dy}{dx}$,  equation  (\ref{quartica.axiais.2}) is given by
\begin{equation}
\label{quartica.axiais.p}
{\Gg}(x,y,p)=a_4 p^4+a_3 p^3+a_2 p^2+a_1 p + a_0 = 0.
\end{equation}

Consider the  \textit{  Lie-Cartan vector field}  $X_{\alpha}$, of class  ${\Cc}^{r-3}$, tangent  to the surface  ${\Gg}=0$
\begin{equation}
\label{campo.liecartan}
X_{\alpha}:={\Gg}_p \frac{\partial}{\partial x}+ p{\Gg}_p \frac{\partial}{\partial y}-({\Gg}_x+p{\Gg}_y) \frac{\partial}{\partial p}.
\end{equation}

The  \textit{axial curvature lines}  are the projections by $\pi: PM \fun M$ restricted to ${\LL}_{\alpha}$, of the integral curves of
  $X_{\alpha}$.
  
  See illustration in Figure  \ref{superficie.lie.cartan}.
 For each $\p \in M-{\Uu}_{\alpha}$  there are  $4$ well defined axial directions, given the four roots of
 equation   (\ref{quartica.axiais.p}).

Two {\it axial configurations} are given: \textit{Principal axial configuration} ${\Pp}_{\alpha} = \{{\Uu}_{\alpha},{\Xx}_{\alpha}\}$ defined by the axiumbilic points   ${\Uu}_{\alpha}$ and by the net ${\Xx}_{\alpha}$ (related to the crossing of principal axial curvature),  in $M-{\Uu}_{\alpha}$ and  \textit{Mean axial configuration} ${\Qq}_{\alpha} = \{{\Uu}_{\alpha},{\Yy}_{\alpha}\}$ defined by the axiumbilic  points ${\Uu}_{\alpha}$ and the net ${\Yy}_{\alpha}$ (related to the crossing of mean axial curvature), in $M-{\Uu}_{\alpha}$.

\begin{figure}[h]
       \centering   
       \fbox{
       \includegraphics[scale=0.35]{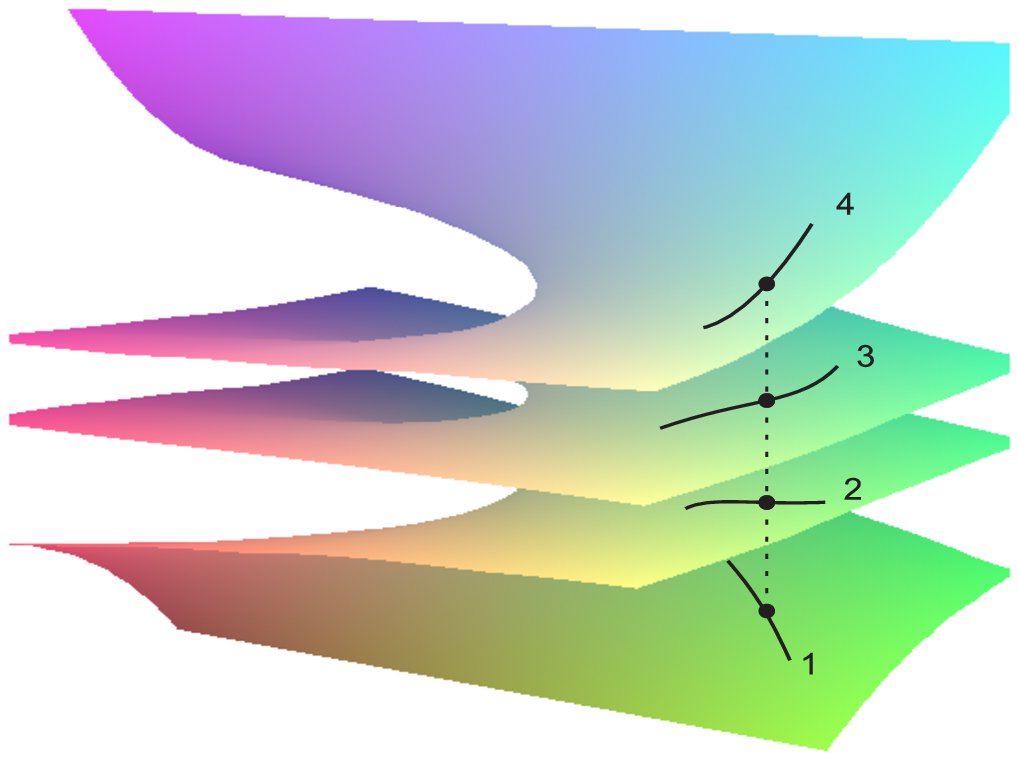}
       \includegraphics[scale=0.35]{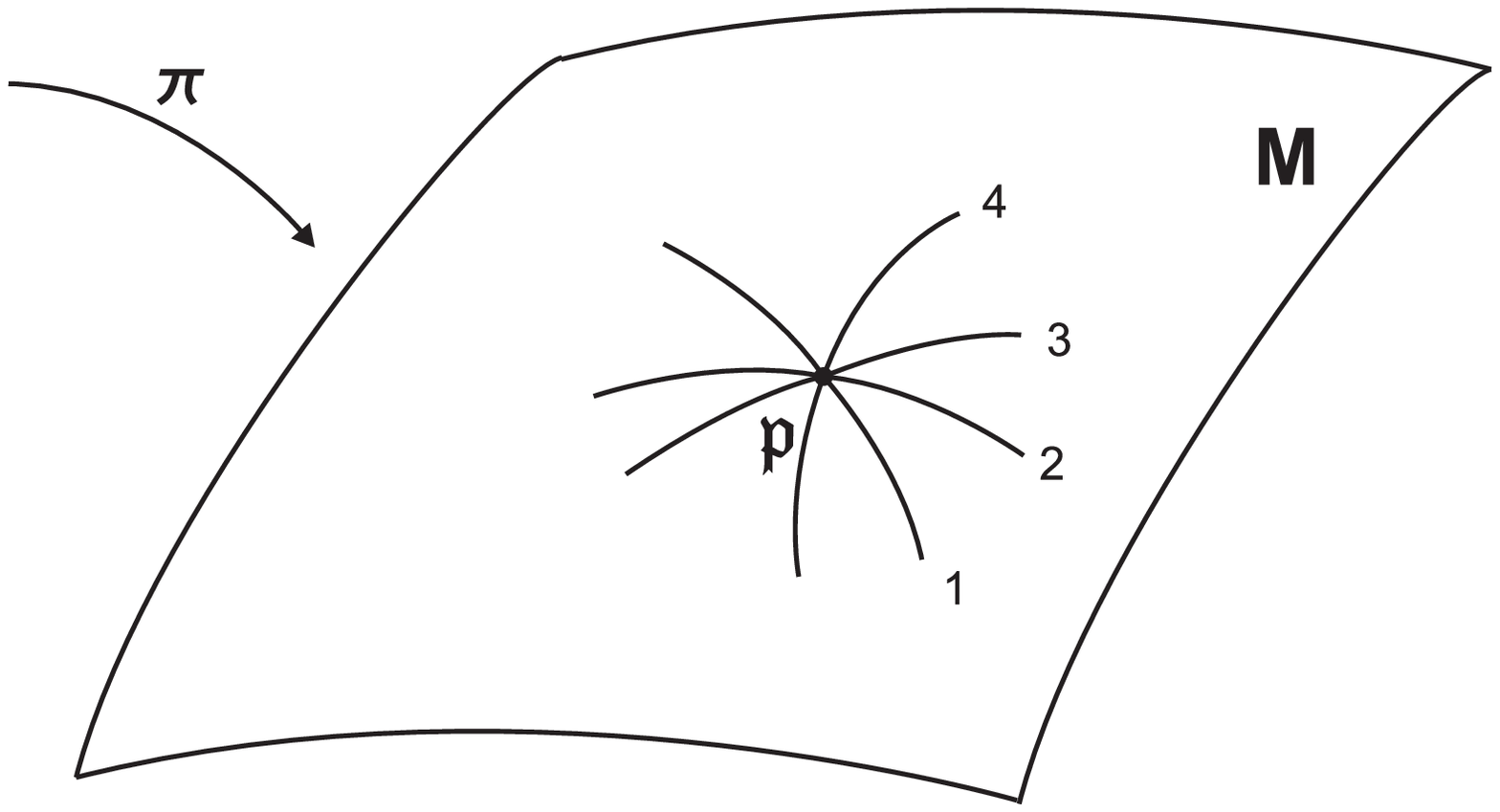}
       }
\caption{Projection on  $M$  of the integral curves of the Lie-Cartan vector field tangent to ${\LL}_{\alpha}$ in a neighborhood of  $\p \in M-{\Uu}_{\alpha}$. For each point in $M$ pass four lines, associated
in pairs to the axis of the ellipse.}
\label{superficie.lie.cartan}
\end{figure}

\section{Differential Equation  of Axial Lines in a Monge Chart }\label{sec:2}

The surface   $M$ will be locally parametrized by a Monge chart near an axiumbilic point $\p$  as follows
$$
\begin{array}{r}
z=R(x,y),\\
w=S(x,y),
\end{array} 
$$
where
\begin{equation}
\label{Rx}
\aligned
R(x,y)=& \frac{r_{20}}{2}x^2+r_{11}xy+\frac{r_{02}}{2}y^2+\frac{r_{30}}{6}x^3+\frac{r_{21}}{2}x^2y+\frac{r_{12}}{2}xy^2+\frac{r_{03}}{6}y^3\\
+ &\frac{r_{40}}{24}x^4+\frac{r_{31}}{6}x^3 y+\frac{r_{22}}{4}x^2 y^2+ \frac{r_{13}}{6}xy^3+\frac{r_{04}}{24}y^4+h.o.t.,\endaligned
\end{equation}

\begin{equation}
\label{S}\aligned
S(x,y)= &\frac{s_{20}}{2}x^2+s_{11}xy+\frac{s_{02}}{2}y^2+ \frac{s_{30}}{6}x^3+\frac{s_{21}}{2}x^2y+\frac{s_{12}}{2}xy^2+\frac{s_{03}}{6}y^3\\
+ &\frac{s_{40}}{24}x^4+\frac{s_{31}}{6}x^3 y+\frac{s_{22}}{4}x^2 y^2+ \frac{s_{13}}{6}xy^3+\frac{s_{04}}{24}y^4+h.o.t.\endaligned
\end{equation}

At the point $(x,y,R(x,y),S(x,y))$, the tangent plane to the surface is generated by $\{t_1,t_2\}$,
where $t_1=(1,0,R_x,S_x)$ and $t_2=(0,1,R_y,S_y)$.
The normal plane is generated by
  $\{N_1,N_2\}$, where $N_1=\frac{\widetilde{N_1}}{|\widetilde{N_1}|}$ and $N_2=\frac{\widetilde{N_2}}{|\widetilde{N_2}|}$ are defined by  $\widetilde{N_1}=(-R_x,-R_y,1,0)$ and $\widetilde{N_2}= t_1 \wedge t_2 \wedge \widetilde{N_1}$.
   Here $\wedge$ is the
   exterior  or \textit{wedge}  product of three vectors in $\R^4$,
   defined by the equation:
$$\det(t_1,t_2,\widetilde{N_1},\bullet)=\lp \widetilde{N_2},\bullet \rp. $$

From the expressions of  $R$ and  $S$ given by  equations (\ref{Rx}) and  (\ref{S}), it follows that:

$$E=1+O(2),  \hspace{1cm} \ F=O(2), \hspace{1cm} \ G=1+O(2),$$
and
$$
\begin{array}{cc}
e_1=r_{20}+r_{30}x+r_{21}y+O(2),&e_2=s_{20}+s_{30}x+s_{21}y+O(2), \\
f_1=r_{11}+r_{21}x+r_{12} y+O(2),&f_2=s_{11}+s_{21}x+s_{12} y+O(2), \\
g_1=r_{02}+r_{12} x+r_{03}y+O(2),&g_2=s_{02}+s_{12} x+s_{03}y+O(2).
\end{array}
$$

The axiumbilic points are defined by    $a_0(x,y)=0$ and $a_1(x,y)=0$. So, in a neighborhood of   $(0,0)$, it follows that
\begin{equation}
\label{a0}
a_0(x,y)= a^0_{00} + a^0_{10} x + a^0_{01} y + O(2)
\end{equation}
and
\begin{equation}
\label{a1}
a_1(x,y)= a^1_{00} + a^1_{10} x + a^1_{01} y + O(2),
\end{equation}
where
\begin{eqnarray*}
a^0_{00} &=& r_{11}(r_{02} - r_{20}) + s_{11}(s_{02} - s_{20}), \\
a^0_{10}&=& r_{21}(r_{02} - r_{20}) + r_{11}(r_{12}-r_{30})+s_{11}(s_{12}-s_{30})+s_{21}(s_{02}-s_{20}), \\
a^0_{01} &=& r_{12}(r_{02} - r_{20}) + r_{11}(r_{03}-r_{21})+s_{11}(s_{03}-s_{21})+s_{12}(s_{02}-s_{20})
\end{eqnarray*}
and
{\small
\begin{eqnarray*}
a^1_{00} &=& (r_{02} - r_{20})^2 + (s_{02} - s_{20})^2-4(r_{11}^2 + s_{11}^2), \\
a^1_{10} &=& 2 (r_{12}-r_{30})(r_{02}-r_{20}) + 2 (s_{12}-s_{30})(s_{02}-s_{20})-8(r_{21}r_{11}+s_{21}s_{11}),\\
a^1_{01} &=& 2 (r_{03}-r_{21})(r_{02}-r_{20}) + 2 (s_{03}-s_{21})(s_{02}-s_{20})-8(r_{12}r_{11}+s_{12}s_{11}).
\end{eqnarray*}
}
Therefore a point $\p$, expressed in a Monge chart by $(0,0)$, is an axiumbilic point when the following relations
hold.

\begin{equation}
\label{axiumbilico}
\left \{
\begin{array}{l}
a^0_{00} = r_{11}(r_{02} - r_{20}) + s_{11}(s_{02} - s_{20}) = 0,\\
a^1_{00} = (r_{02} - r_{20})^2 + (s_{02} - s_{20})^2-4(r_{11}^2 + s_{11}^2)= 0.\hspace{2cm}
\end{array}
\right .
\end{equation}

Algebraic manipulations of the equations above, see \cite{gutierrez1}, show that
  $(0,0)$ is  an axiumbilic  point when the following equations hold

\begin{equation}
\label{axiumbilico.simples}
\left \{
\begin{array}{lcl}
2r_{11}=(s_{02}-s_{20}), \\
2s_{11}=-(r_{02}-r_{20},)
\end{array}
\right.
\hspace{2cm} \mbox{or}\hspace{2cm}
\left \{
\begin{array}{lcl}
2r_{11}=-(s_{02}-s_{20}), \\
2s_{11}=(r_{02}-r_{20}).
\end{array}
\right.
\end{equation}

\begin{obs}
Let  $r_{02}=r_{20}+r$ and $s_{02}=s_{20}+s$, $\rho^2=r^2_{11} + s^2_{11}$. Then the condition for  $(0,0)$  to be an axiumbilic point, see equation  (\ref{axiumbilico}),  is given by
\begin{equation}
\label{axiumbilico.rs}
\left \{
\begin{array}{l}
r_{11}\cv r + s_{11}\cv s = 0, \\
r^2 + s^2 = 4\rho^2 .\hspace{2cm}
\end{array}
\right .
\end{equation}
\end{obs}
\nin These  condition for being an axiumbilic point
can be interpreted as the intersection of  a circle and a straight line in the plane   $(r,s)$. The intersections are given by
\begin{equation}
\label{axiumbilico.simples.rs}
\left \{
\begin{array}{lcl}
r_{11}=\frac{s}{2}, \\
s_{11}=-\frac{r}{2},
\end{array}
\right .
\hspace{2cm} \mbox{or}\hspace{2cm}
\left \{
\begin{array}{lcl}
r_{11}=-\frac{s}{2}, \\
s_{11}=\frac{r}{2},
\end{array}
\right .
\end{equation}
\nin and therefore
 equation  (\ref{axiumbilico.simples.rs})
is another
form
of equation (\ref{axiumbilico.simples}).

Let
$$\aligned \alpha_1=&s_{12} -s_{30}+2 r_{21},\; \alpha_2= r_{30}-r_{12}+2s_{21},\\
 alpha_3=& s_{03}-s_{21}+2r_{12},\;  \alpha_4= r_{21}-r_{03 }+2s_{12}.\endaligned$$

The discussion above
is synthesized  in
the following lemma.
\begin{lema}
\label{lema.eq.axial.monge}
Let $\p$ be an axiumbilic point with coordinates  $(0,0)$
in a Monge chart.
The differential equation of axial lines in a neighborhood of $(0,0)$ is given by
\begin{equation}
\label{eq.axial.monge}
\tilde{a}_0(x,y) (dx^4-6dx^2dy^2+dy^4)+\tilde{a}_1(x,y)(dx^2-dy^2)dx dy + H(x,y,dx,dy)=0,
\end{equation}
where
\begin{equation}
\label{eq:a0a1}
\aligned
 \tilde{a}_0(x,y) = & \frac{1}2 (r\alpha_1+ s\alpha_2)x+\frac 12(r\alpha_3+ s\alpha_4)y + a^0_{20} x^2 + a^0_{11} x y + a^0_{02} y^2 ,\\
\tilde{a}_1(x,y) =& 2(s\alpha_1-r\alpha_2)x+2(s\alpha_3-r\alpha_4)y + a^1_{20} x^2 + a^1_{11} x y + a^1_{02} y^2\endaligned
\end{equation}
and  $H$ contains terms of order
 greater  than
 or equal to    $3$ in $(x,y)$.

\end{lema}

With the notation
in equation (\ref{eq.axial.monge}),
the {\it condition of transversality} between the curves
  $a_0=0$ and $a_1=0$ is given by
\begin{equation*}
\label{determinante}
\left |
\begin{array}{cc}
a^0_{10} & a^0_{01} \\
a^1_{10} & a^1_{01}
\end{array}
\right |
\neq 0.
\end{equation*}

\nin The  determinant above has the following expression:
\begin{equation*}
 [\alpha_2\alpha_3-\alpha_1\alpha_4]\cv (r^2+s^2),
\end{equation*}
where $r=r_{02}-r_{20}$ and $s=s_{02}-s_{20}$. If $(r^2+s^2)$ is zero, it follows that $a^0_{10}=a^0_{01}=a^1_{10}=a^1_{01}=0$, and therefore the matrix
$$
\left [
\begin{array}{cc}
a^0_{10} & a^0_{01} \\
a^1_{10} & a^1_{01}
\end{array}
\right ]
$$
is identically zero.
Thus the axiumbilic points with $r=s=0$  form a set of codimension at least four.

 Therefore, the {\it condition of transversality}, supposing $r^2+s^2\ne 0$, is given by:

\begin{equation}
\label{determinante2}
T:=  \alpha_2\alpha_3-\alpha_1\alpha_4   \neq 0.
\end{equation}

Long, but straightforward calculations show
 that the condition
  (\ref{determinante2}) is invariant by positive rotations in the tangent and in the normal plane.

\begin{lema}\label{lema:rota} Consider the quartic differential equation
$$ (a_{10}x+a_{01}y)(dx^4-6dx^2dy^2+dy^4)+(b_{10}x+b_{01}y)dxdy(dx^2-dy^2)=0.$$
Consider a rotation $x=\cos\theta u+\sin\theta v, \; y=-\sin\theta u+\cos\theta v$,
where $\theta $ is a real root of   the equation
{\small
$$ -a_{01} t^5+(a_{10}-b_{01})t^4+(6 a_{01}+b_{10})t^3+(b_{01}-6 a_{10})t^2-(a_{01}+b_{10})t+a_{10}=0, \;t=\tan\theta.$$
}
Then it follows that

$$ \bar{a_{01}}v(du^4-6du^2dv^2+dv^4)+(\bar{b_{10}}u+\bar{b_{01}}v)dudv(du^2-dv^2)=0.$$
where $\bar{a_{01}}=\bar{a_{01}}(a_{10},a_{01},b_{10}, b_{01} ,\theta)$,  $\bar{b_{10}}=\bar{b_{10}}(a_{10},a_{01},b_{10}, b_{01}, \theta)$ and
 $\bar{b_{01}}=\bar{b_{01}}(a_{10},a_{01},b_{10}, b_{01}, \theta)$.
\end{lema}

\begin{proof} The result follows from straightforward calculations. Observe that when $a_{01}=0$   a rotation of angle $\pi/2$ is sufficient to obtain the result stated.
\end{proof}

\begin{prop}
\label{prop.axial}
Let $\p$ be an axiumbilic point. Then there exists a Monge chart and a homotety in   $\R^4$ such that the differential equation of axial lines is given by
\begin{equation}
\label{axial.monge}
y(dy^4-6dx^2dy^2+dx^4)+(a x+b y)dxdy(dx^2-dy^2) + H(x,y,dx,dy)=0
\end{equation}
where  $H$ contains terms of order greater than  or equal to  $2$ in  $(x,y)$. Moreover, the axiumbilic point  $\p$ is transversal if and only if  $a \neq 0$.
\end{prop}

\begin{proof} Consider a parametrization $X(x,y)=(x,y, R(x,y), S(x,y))$  given by equations \eqref{Rx} and \eqref{S} such that $0$ is an axiumbilic point. By  equation \eqref{eq:a0a1}   it follows that:
$$\aligned a_0(x,y)=&    \frac{1}2 (r\alpha_1+ s\alpha_2)x+\frac 12(r\alpha_3+ s\alpha_4)y  +O(2),\\
a_1(x,y) =&  2(s\alpha_1-r\alpha_2)x+2(s\alpha_3-r\alpha_4)y +O(2).
\endaligned
$$

By an appropriate choice of the rotation in the plane $\{x,y\}$ given by Lemma \ref{lema:rota} and a homotety in $\mathbb R^4$, it is possible to make $2a_{10}=r\alpha_1+ s\alpha_2=0$ and, when
$(\alpha_1\alpha_4-\alpha_2\alpha_3)(r^2+s^2)\ne 0$, also $a_{01}=\frac 12(r\alpha_3+ s\alpha_4)=1$.
So the result is established,   $a=\frac{4(s\alpha_1-r\alpha_2)}{r\alpha_3+ s\alpha_4}$ when
$r\alpha_1+ s\alpha_2=0 $ and $b=\frac{4(s\alpha_3-r\alpha_4)}{r\alpha_3+ s\alpha_4 }$.
If $r\ne 0$ it follows that $a=-\frac{4(r^2+s^2)\alpha_2}{r(r\alpha_3+s \alpha_4)}$
and $a=\frac{4\alpha_1}{\alpha_4}$  when $s\ne 0$ and $r=0$.
\end{proof}

\begin{obs}
\label{axial.p.obs}
Let $p=\frac{dy}{dx}$. Then the differential equation    (\ref{axial.monge})    is given by:

\begin{equation}
\label{axial.p}
y(p^4-6p^2+1)+(a x+b y)p(1-p^2) + H(x,y,p) = 0,
\end{equation}
where $H$ contains terms of order greater than or  or equal to $2$ in $(x,y)$.
\end{obs}

\section{ Axial configuration in the neighborhood of axiumbilic points}
\label{sec:3}

Let $\p$ be an axiumbilic point whose neighborhood is parametrized by a Monge chart following the notation established in Section  \ref{sec:2}. When it is a transversal axiumbilic point, which is determined by transversal intersection of the curves  $a_0=0$ and $a_1=0$  (see equation (\ref{determinante.transversal})),  
it results from Proposition   \ref{prop.axial} and Remark  \ref{axial.p.obs} that
the differential equation of axial lines is given by

\begin{equation}
\label{sup.cartan}
{\Gg}(x,y,p)= y(p^4-6p^2+1)+(ax+by)p(1-p^2)+H(x,y,p) = 0,
\end{equation}
where
$H(x,y,p)$ contains higher order terms greater or equal to    $2$ in $(x,y)$.

The  Lie-Cartan surface $\LL_{\alpha}$ in $PM$ is defined  implicitly by

\begin{equation}
\label{implicit}
{\Gg}(x,y,p)=0.
\end{equation}
In the case that $\p$ is a transversal axiumbilic point the surface defined above
is regular and of class  ${\Cc}^{r-2}$
in the neighborhood of the projective axis
 $p$.

In the coordinates
 $(x,y,p)$, the Lie-Cartan vector field $X$, is of class ${\Cc}^{r-3}$, (equation (\ref{campo.liecartan})):
\begin{equation}
\label{liecartan}
X:={\Gg}_p \frac{\partial}{\partial x}+ p{\Gg}_p \frac{\partial}{\partial y}-({\Gg}_x+p{\Gg}_y) \frac{\partial}{\partial p}
\end{equation}
and the projections of the integral curves of  $X \bigg |_{{\Gg}=0}$ are the axial lines in a neighborhood of  $\p$ (Figure \ref{superficie.lie.cartan}).

Restricted to the projective axis $p$ the Lie-Cartan vector field is given by

\begin{equation*}
X = -p[(p^4-6p^2+1)+(1-p^2)(a+bp)]  \frac{\partial}{\partial p}.
\end{equation*}

Therefore, the singular points of the Lie-Cartan vector field in the projective line are
given by the equation:
\begin{equation}
\label{P}
P(p)=pR(p)=p[(p^4-6p^2+1)+(1-p^2)(a+bp)]=0.
\end{equation}

The discriminant of $R(p)=(p^4-6p^2+1)+(1-p^2)(a+bp)$ is

\begin{equation}\label{eq:delta}\aligned
    \Delta(a,b)=&
16a^5+4( b^2+68)a^4+16( b^2+144)a^3\\
-&8(b^2-80)(16+b^2)a^2+96(16+b^2)^2 a+ 4(16+b^2)^3 .
\endaligned
\end{equation}

Furthermore, $R(\pm 1)=-4$, $R(0)=1+a$ and $\lim_{p\fun \pm \infty} R(p)=+\infty$, thus $R$ has at least two simple real roots,
one
is
less than $-1$ and the other
is
greater
 than $1$.

The derivative of $X$ at  $(0,0,p)$ is given by:
$$
DX(0,0,p) =
\left [
\begin{array}{ccc}
a(1-3p^2) & 4p^3+b(1-3p^2)-12p & 0 \\
a(1-3p^2)p & p[4p^3+b(1-3p^2)-12p] & 0 \\
0 & 0 & -P'(p)
\end{array}
\right ]
$$
whose eigenvalues are $0$ and
 
$$
\begin{array}{l}
\lambda_1(p)=a(1-3p^2)+p[4p^3+b(1-3p^2)-12p], \\
\lambda_2(p)= -P'(p). \\
\end{array}
$$

Recall that $P(p)=pR(p)$, and so $P'(p)=R(p)+pR'(p)$. Therefore at the roots of     $R$, it follows that $-P'(p)=-pR'(p)$. Also, as $\pm 1$ are not roots of   $R $, it follows that

$$a=\frac{(-p^4+6p^2-1)+bp(1-p^2)}{1-p^2}.$$

Substituting the equation above into  the expression of   $\lambda_1(p)$,  $p$ being a    root  of   $R(p)$ (singular points of  $X$), it follows that
$$
\left \{
\begin{array}{l}
\lambda_1(p)=\frac{(p^2+1)^3}{(p^2-1)}, \\
\lambda_2(p)= -pR'(p). \\
\end{array}
\right .
$$
Therefore,  the   eigenvalues of $DX$,
 on the tangent space to $\Gg =0$,
 are as follows:
\begin{equation}
\label{p0}
p_0=0 :\ \
\left \{
\begin{array}{l}
\lambda_1= a, \\
\lambda_2=-(a+1),
\end{array} \ \ \ \
\right .
\end{equation}
\begin{equation}
\label{pi}
p_i\neq 0 :\ \
\left \{
\begin{array}{l}
\lambda_1=\frac{(p_i^2+1)^3}{(p_i^2-1)}, \\
\lambda_2= -p_i R'(p_i). \\
\end{array} \ \ \ \
\right .
\end{equation}

The eigenspace  associated to the eigenvalue $\lambda_1$ is  transversal to the axis $p$ and the eigenvalue  $\lambda_2$ has the projective axis as the associated eigenspace.

In  \cite{sotogarcia1} the axial configuration near an axiumbilic point was established in the following situation:
\begin{itemize}
\item $\Delta(a,b) < 0$,
\item $\Delta (a,b)> 0$,\; $a<0, \; a \neq -1$,
\item $\Delta(a,b) > 0$,\; $a>0$.
\end{itemize}

When $\Delta(a,b) < 0$, $R$ has two simple real roots, and the Lie-Cartan
vector field
has three hyperbolic saddles in the projective axis. This axiumbilic point is called of type $E_3$.

When $\Delta(a,b) > 0$, $a<0, \ a \neq -1$,   $R$ has four simple real roots, and the Lie-Cartan vector field has 5 singular points in the projective line. Four are hyperbolic saddles and one is a hyperbolic node.   This axiumbilic point is called of type $E_4$.

When  $\Delta (a,b)> 0$, $a>0$, the  Lie-Cartan vector field has 5 hyperbolic saddles in the projective line. This axiumbilic point is called of type $E_5$.

In Figure  \ref{superficies} the Lie-Cartan surfaces and the integral curves of the Lie-Cartan vector field are sketched in the three cases  $E_3$, $E_4$ and $E_5$.
The projections of the integral curves by   $\pi: PM \fun M$ are the axial lines near the axiumbilic points (see Figure  \ref{configuracoes})
 $E_3$, $E_4$ and $E_5$.

\begin{figure}[h]
       \centering  
       \fbox{
       \includegraphics[scale=0.55]{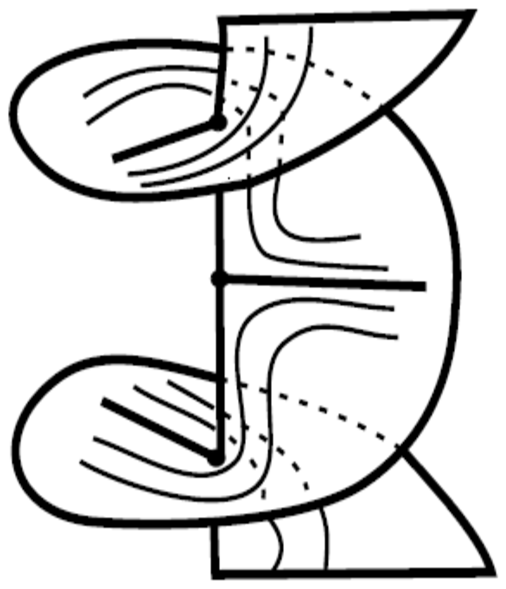}\;\;\;\;
       \includegraphics[scale=0.53]{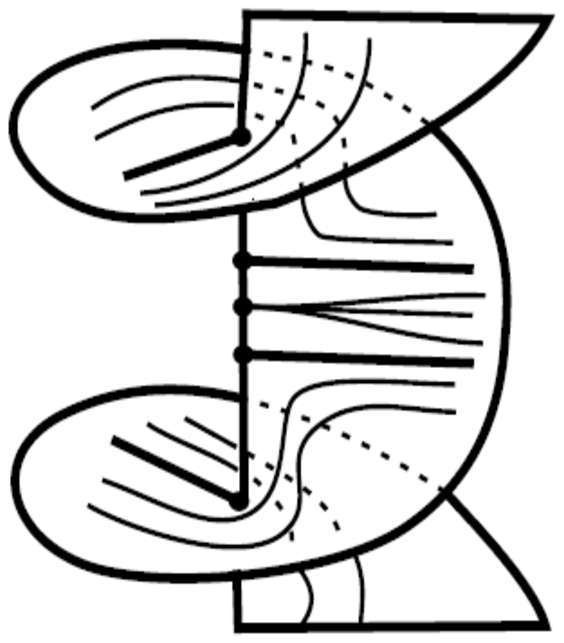}\;\;\;\;
       \includegraphics[scale=0.32]{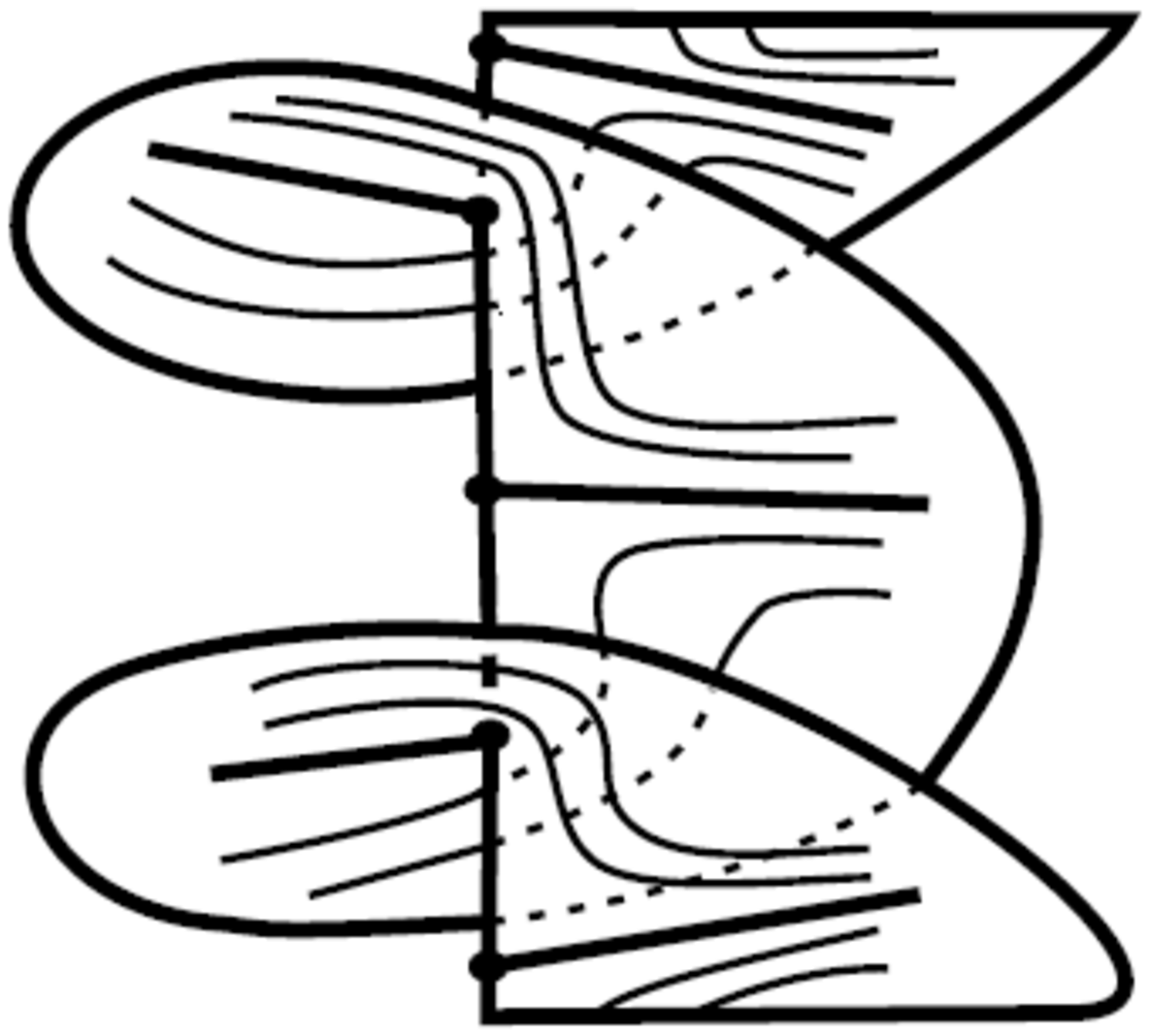} }
\caption{ Lie-Cartan vector field and its integral curves in the cases $E_3$, $E_4$ and $E_5$.}
\label{superficies}
\end{figure}

\begin{figure}[h]
       \centering  %
       \fbox{
       \includegraphics[scale=0.35]{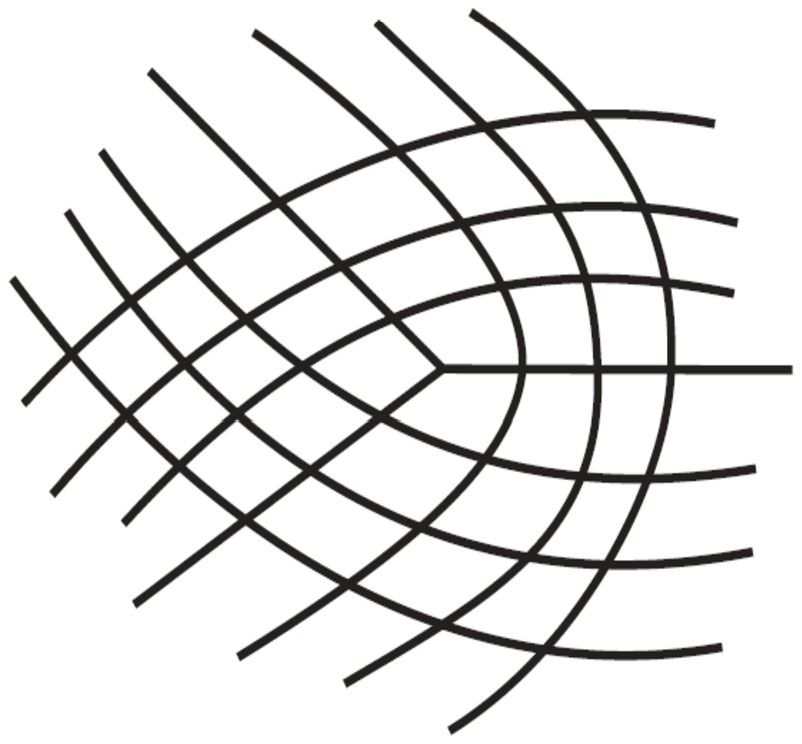}
       \includegraphics[scale=0.32]{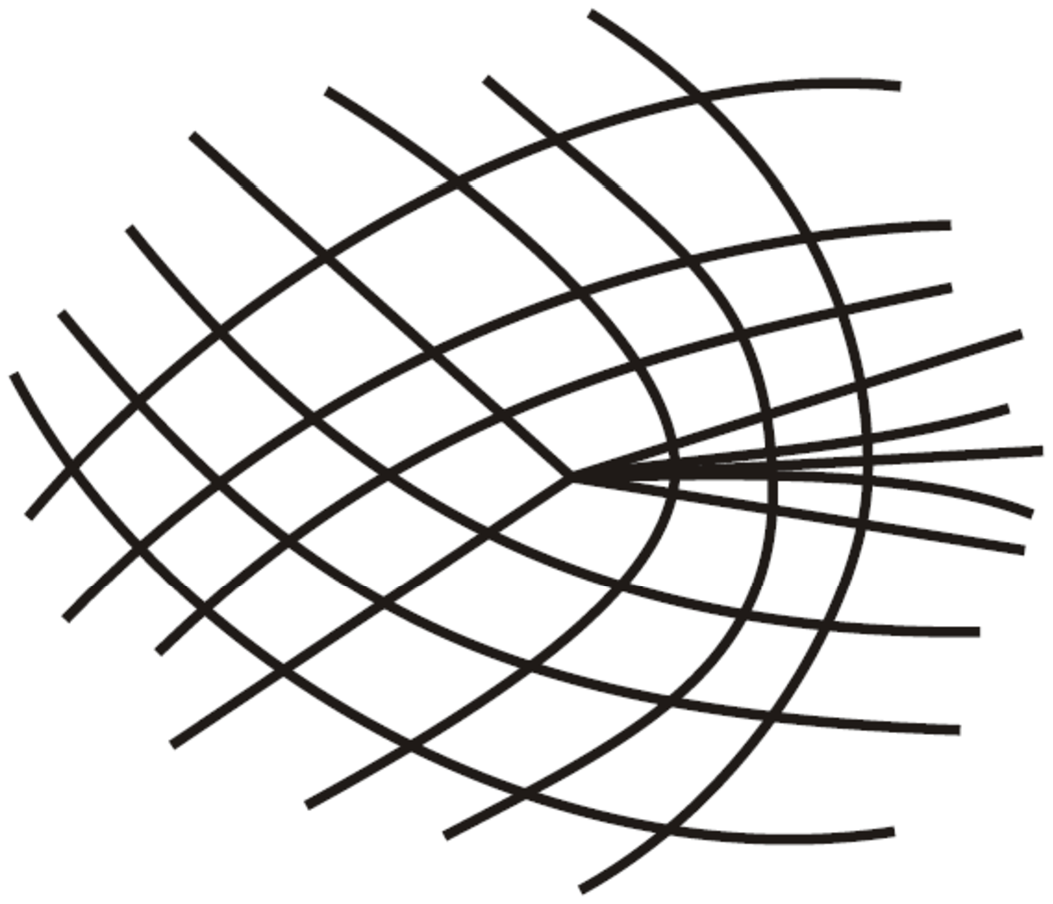}\;\;\;\;
       \includegraphics[scale=0.32]{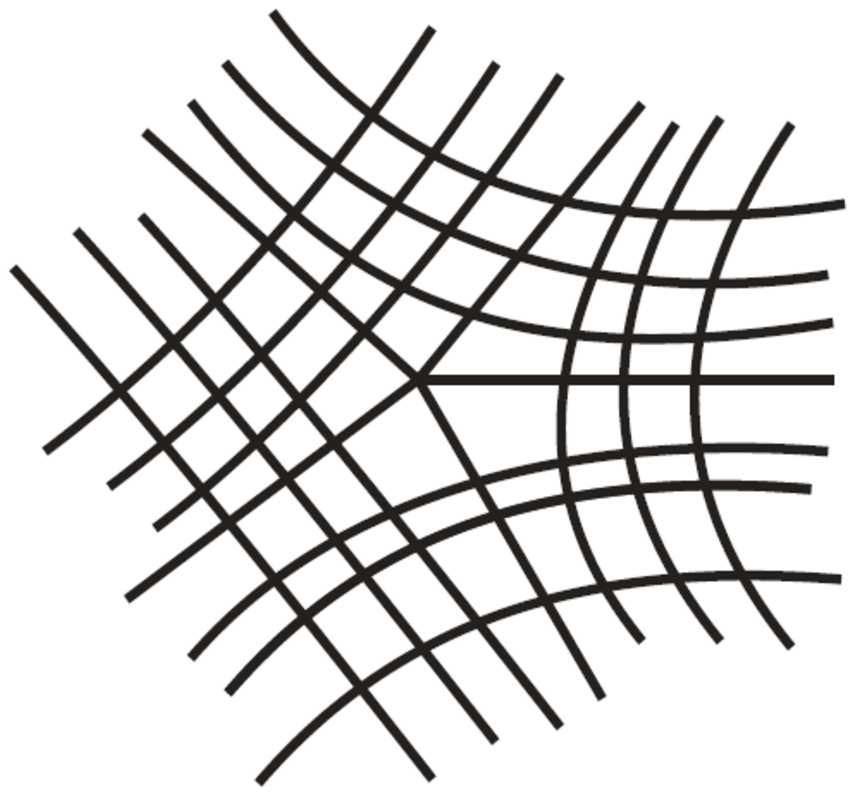}}
        \fbox{
       \includegraphics[scale=0.30]{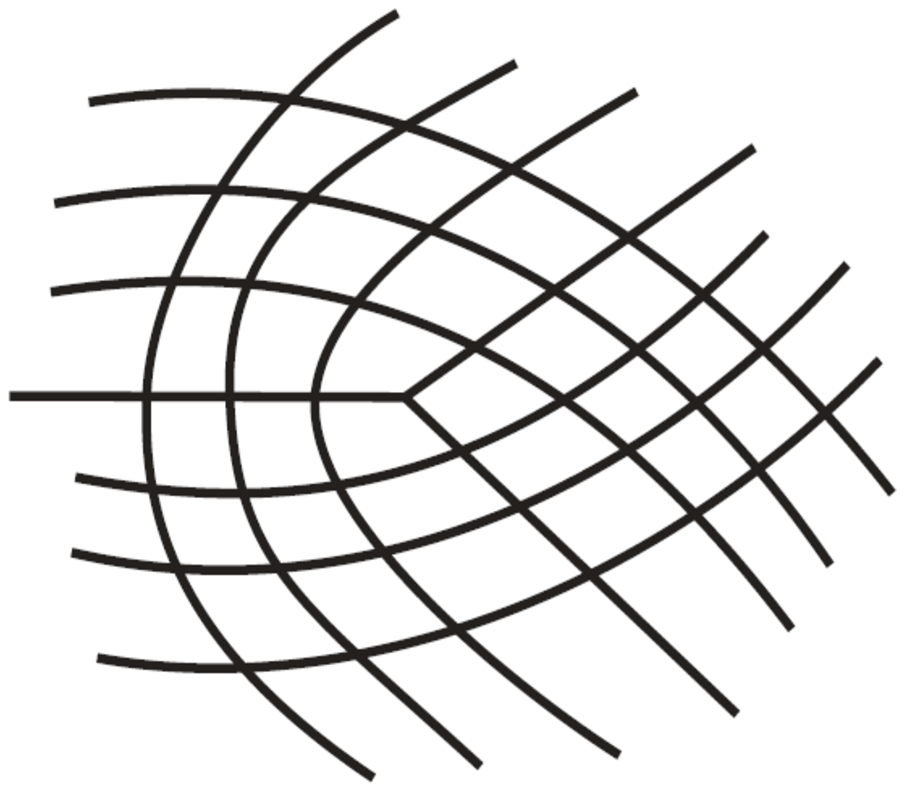}\;\;\;\;
       \includegraphics[scale=0.30]{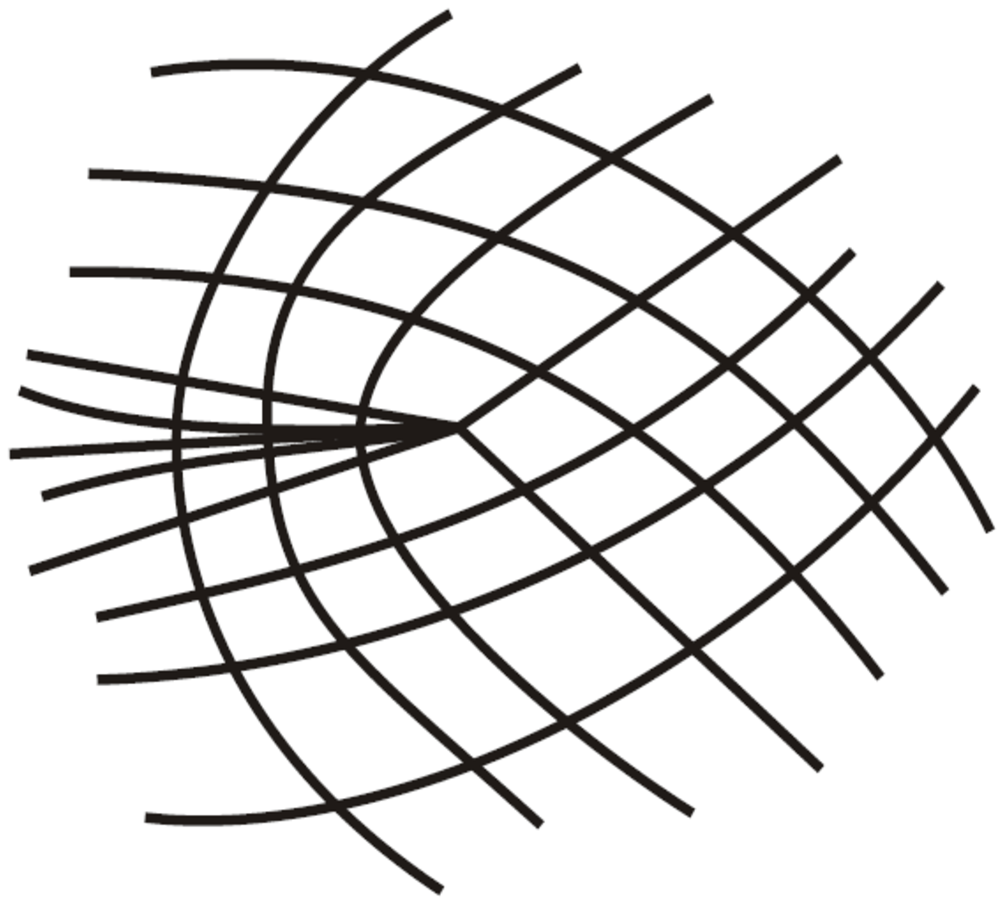}\;\;\;\;\;\;
       \includegraphics[scale=0.30]{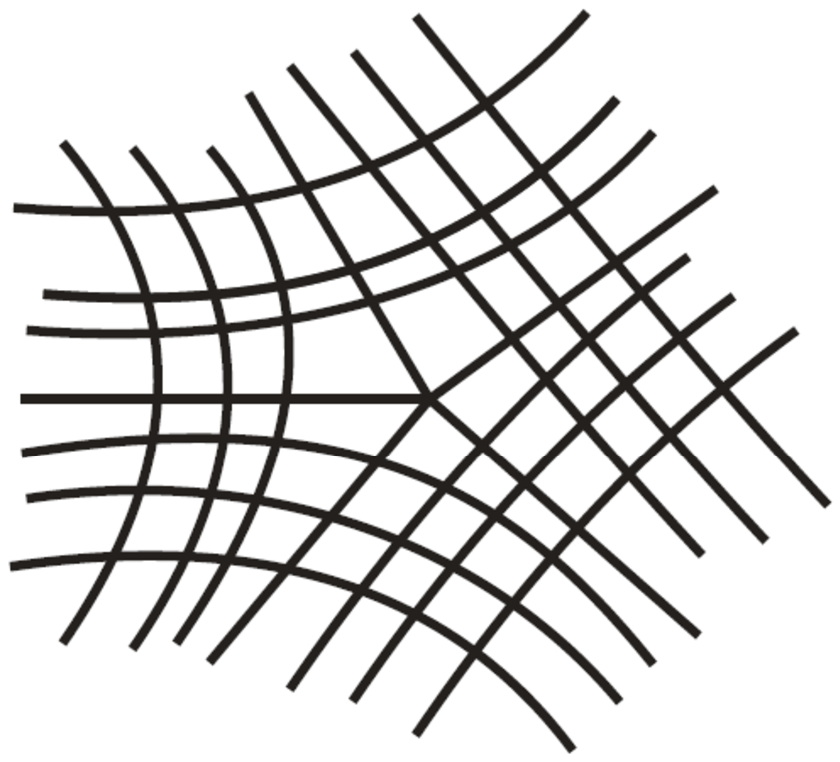}}
\caption{Axial Configurations near Axiumbilic Points $E_3$ (left), $E_4$ (center) and $E_5$ (right).}
\label{configuracoes}
\end{figure}

For an immersion $\alpha$  of a surface  $M$ into
   $\mathbb R^4$, the  axiumbilic singularities ${\Uu}_{\alpha}$
    and the lines of  axial  curvature are assembled into
     two {\it axial configurations}:
the  {\it  principal axial configuration}
${\Pp}_{\alpha}= \{ {\Uu}_{\alpha},\; {\Xx}_{\alpha} \}$ and the  {\it mean  axial configuration}  ${\Qq}_{\alpha}= \{ {\Uu}_{\alpha},\; {\Yy}_{\alpha}\}.$
 
An immersion $\alpha\in {\Ii}^r$ is said to be {\it Principal Axial Stable} if it has a $C^r$ neighborhood ${\mathcal V}(\alpha)$ such that, for any $\beta\in   {\mathcal V}(\alpha)$ there exists a homeomorphism $h: M \to M $ mapping ${\Uu}_\alpha$ onto  ${\Uu}_\beta$ and mapping the integral net of ${\Xx}_\alpha$ onto that of ${\Xx}_\beta$.
Analogous definition is given for {\it Mean Axial Stability}.

In Proposition
  \ref{axi.estaveis}  are described the axiumbilic points  which are axial stable. In
 Figure  \ref{diagrama 3}
 are sketched
 the curves   $\Delta(a,b)=0$, $a=-1$ and $a=0$
 in the plane $a ,
 b$,
 which bound the open regions
 corresponding to  the
 three  types of axiumbilic  points of axial stable type.

\begin{prop}[\cite{sotogarcia1}, \cite{sotogarcia} p. 209]
\label{axi.estaveis}
Let $\p$ be an axiumbilic point of  $\alpha \in {\mathcal{I}}^r, \ r\geq 5$. Then, $\alpha$ is locally  principal
axial stable and locally mean axial stable at  $\p$ if and only if $\p$ is of type $E_3,\ E_4$ or $E_5$.
 The curve $\Delta(a,b)=0$   has three connected components, is contained in the region $a\leq -1$ and it is regular outside the points $(-\frac{27}2,\pm\frac{5\sqrt{5}}2)$ which are of cuspidal type.

\end{prop}
\begin{figure}[h]
       \centering  
       \fbox{
       \includegraphics[scale=0.2]{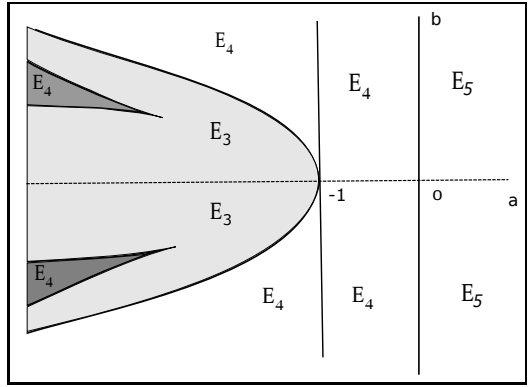}}
\caption{Diagram of stable axiumbilic points, $E_3$, $E_4$ and $E_5$.}
\label{diagrama 3}
\end{figure}
\begin{proof} The function $\Delta(a,b)$ defined by equation \eqref{eq:delta} is symmetric in $b$. The polynomials
$\Delta(a,b)$ and $\frac{\partial \Delta}{\partial b}$ in the variable $b$ has resultant equal to $274877906944 (1+a)  (a^2+8a+32)^2 a^{16}(2a+27)^6$.

The critical points    $p_{\pm}=(-\frac{27}2,\pm \frac{5\sqrt{5}}2)$ of $\Delta$ are contained in  $\Delta(a,b)=0$.

 Near the point $p_+$  it follows that:
{\small
$$\aligned  \Delta(a,b)=& -54675\left[ \left(a+\frac{27}2 \right)^2+5 \left( b-\frac{5\sqrt{5}}2 \right)^2+2\sqrt{5} \left(a+\frac{27}2 \right) \left(b-\frac{5\sqrt{5}}2 \right)\right]\\
+&h.o.t.\endaligned $$
}
Further analysis shows $p_\pm$ are Whitney cuspidal points.

Also the curve  $\Delta(a,b)=0$ is contained in the region $a\leq -1$ and near $(-1,0)$ it is given by $a=-\frac 1{20}b^2+O(3)$. In fact,  for $a>-1$ all the roots of $\Delta(a,b)$ are complex.

By the classification of axiumbilic points $E_3$, $E_4$ and $E_5$ by the sign of $\Delta(a,b)$ and of $a$, the diagram of stable axiumbilic points, see   \cite{sotogarcia1}, \cite{sotogarcia} p. 209, is as shown in Fig. \ref{diagrama 3}.
\end{proof}

\subsection{The axiumbilic point $E^1_{34}$}

\begin{defn}
\label{defn.e34}
Let $\alpha: M  \fun \R^4$ be an immersion of class ${\Cc}^r,  r\geq 5,$ of a smooth and oriented surface.
An axiumbilic point   $\p$ is said to be of type  $E^1_{34}$
if $a$ defined in Proposition \ref{prop.axial} does not vanish
and:
\begin{enumerate}[$i)$]
\item $\Delta(a,b)=0$, $(a,b)\ne (-1,0)$  and $(a,b)\ne (-\frac{27}{2},\pm \frac 52\sqrt{5})$, or
 \item $b\neq 0$ if $a=-1.$
\end{enumerate}
\end{defn}

\begin{prop}
\label{proposicao.e3434}
Let $\alpha: M  \fun \R^4$ be an immersion of class ${\Cc}^r,\ r\geq 5$ of a smooth and oriented surface having an axiumbilic point   $\p$  of type  $E^1_{34}$. Then the axial configuration of   $\alpha$ in a neighborhood of  $\p$ is as shown in Figure  \ref{e34}.
\begin{figure}[ht]
       \centering  
       \fbox{
       \includegraphics[scale=0.7]{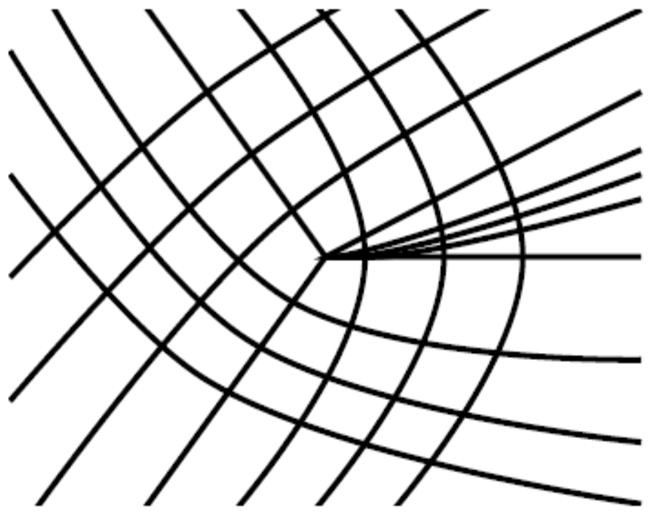}\;\;\;\;\;\;\;\;\;\;\;\;\;\;\;
       \includegraphics[scale=0.7]{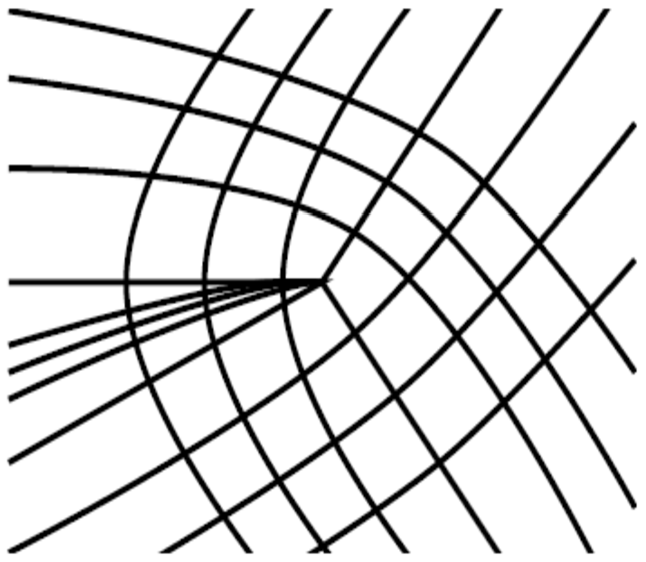}}
\caption{Axial Configurations in a neighborhood of an axiumbilic point of type   $E^1_{34}$.}
\label{e34}
\end{figure}
\end{prop}
\begin{proof}
Since the condition of transversality ($a\ne 0$) is preserved at an axiumbilic point of type $E^1_{34}$
the implicit surface defined by equation (\ref{implicit}) is regular in a neighborhood of the projective line.
From
the hypotheses
 $\Delta(a,b)=0$, $(a,b)\ne (-1,0)$  and $(a,b)\ne (-\frac{27}{2},\pm \frac 52\sqrt{5})$
 or
  $b\neq 0$, if $a=-1$,
  the polynomial $P(p)=p[(p^4-6p^2+1)+(1-p^2)(a+bp)]=pR(p)$, which defines the singularities of the Lie-Cartan vector field,
  has
  one
   double root
   and three real simple roots.
	
\nin With no loss of generality, we can consider the case $a=-1$ and $b\neq 0$, where $p=0$ is a double root of the polynomial $P(p)$. In this case, we have $P(p)=p^2(p^3-b p^2-5 p+b)$.	

\nin The   eigenvalues of $DX$ at $(0,0,p)$   are given by:
 $ \lambda_1=4p^4-3bp^3 -9 p^2+bp-1$ and $\lambda_2=  p(-5p^3+4bp^2+15p-2b)$.

\nin Therefore, at the singular points $(0,0,p)$, $p\neq 0$,  of $X$ it follows that:
 $\lambda_1=\frac{(p^2+1)^3}{p^2-1}$ and $ \lambda_2=-\frac{p^2(p^4+2p^2+5)}{p^2-1} $.
Then, $\lambda_1\lambda_2<0$ when $p\ne 0$
and these three singular points of $X$  are  hyperbolic saddles.
At  $p=0$, double root of $P$, it follows that $\lambda_1=-1,\; \lambda_2=0$. Recall that the eigenspace associated to $\lambda_1$ is transversal to the axis $p$ and that one associated to $\lambda_2$ is the projective axis itself.

 \nin Since ${\Gg}_y(0,0,0)=1$, it follows from the Implicit Function Theorem that
$y(x,p)= x p+O(3)$ is defined in a neighborhood of $(0,0,0)$ such that ${\Gg}(x,y(x,p),p)=0$.
In this case, the  Lie-Cartan vector field  in the chart  $(x,p)$ is given by:

\begin{equation}
\left \{
\begin{array}{l}
\label{lie.cartan.restrito}
\dot{x}=  -x+b x p+O(3) \\
\dot{p}= -b p^2 + O(3)
\end{array}
\right .
\end{equation}
with  $b\neq 0$.
Therefore, $(0,0,0)$ is a quadratic saddle-node   with the center manifold tangent to the projective line. The phase portrait is sketched in   Figure  \ref{cartan.e34}, and the projections of the integral curves are the axial lines shown in Figure   \ref{e34}.
\begin{figure}[h]
       \centering  
       \fbox{
       \includegraphics[scale=0.5]{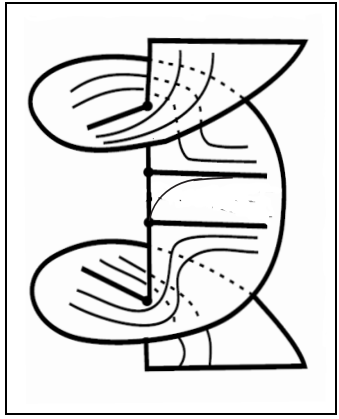}}
\caption{Integral curves of  $X|_{{\Gg}=0}$ in the neighborhood of the projective line in the case of an axiumbilic point of type $E^1_{34}$}
\label{cartan.e34}
\end{figure}

\nin When $(a,b)\ne (-1,0)$,   $(a,b)\ne (-\frac{27}{2},\pm \frac 52\sqrt{5})$  and
$\Delta(a,b)= 0$ the polynomial $P(p)=p[(p^4-6p^2+1)+(1-p^2)(a+bp)]$ has a double root $p_0\neq 0$ and three real simple roots.
This case is reduced to the case when $p=0$ is a double root, making an appropriated rotation of coordinates in the plane $\{x,y\}$ so that the double root $p_0$ is, in the new coordinates, located at  $p=0$. 
\end{proof}

\begin{prop}
\label{proposicao.e34}
Let $\alpha \in {\Ii}^r$, $r\geq 5$, be an immersion such that   $\p$
is axiumbilic point of type $E^1_{34}$.
Then, there is a neighborhood   $V$ of $\p$, a neighborhood  ${\Vv}$ of
 $\alpha$ and a function ${\Ff}: {\Vv}\fun \R$ of class  ${\Cc}^{r-3}$ such that
 for each  $\mu \in {\Vv}$ there is an unique axiumbilic point
  $\p_{\mu} \in V$ such that:
\begin{enumerate}[$i)$]
	\item $d {\Ff}_{\alpha} \neq 0$,
	\item ${\Ff}(\mu)<0$ if and only if  $\p_{\mu}$ is axiumbilic point of type $E_3$,
	\item ${\Ff}(\mu)>0$ if and only if $\p_{\mu}$ is axiumbilic point of type $E_4$,
	\item ${\Ff}(\mu)=0$ if, and only if,  $\p_{\mu}$ is axiumbilic point of type  $E^1_{34}$.
\end{enumerate}
\end{prop}
\begin{proof}
Since $\p$ is a transversal axiumbilic point of   $\alpha$, the existence of the neighborhoods  ${\Vv}$ and $V$ follows from the Implicit Function Theorem. For
$\mu \in {\Vv}$ with an axiumbilic point $\p_{\mu} \in V$, after a rigid motion $\Gamma_{\mu}$ in  $\R^4$, locally the immersion $\mu \in {\Vv}$ can be parametrized in terms of a Monge chart $(x,y,R_{\mu}(x,y),S_{\mu}(x,y))$, with the origin being the axiumbilic point $p_{\mu}$ and
\begin{equation*}
\aligned
R_{\mu}(x,y)=& \frac{r_{20}(\mu)}{2}x^2+r_{11}(\mu)xy+\frac{r_{02}(\mu)}{2}y^2+\frac{r_{30}(\mu)}{6}x^3+\frac{r_{31}(\mu)}{2}x^2y\\
+&\frac{r_{13}(\mu)}{2}xy^2+\frac{r_{03}(\mu)}{6}y^3+ h.o.t.,\endaligned
\end{equation*}
\begin{equation*}\aligned
S_{\mu}(x,y)= & \frac{s_{20}(\mu)}{2}x^2+s_{11}(\mu)xy+\frac{s_{02}(\mu)}{2}y^2+
\frac{s_{03}(\mu)}{6}x^3+\frac{s_{21}(\mu)}{2}x^2y\\
+&\frac{s_{12}(\mu)}{2}xy^2+\frac{s_{03}(\mu)}{6}y^3+ h.o.t.\endaligned
\end{equation*}

For $\mu$, performing rotations and
 homoteties
  as described in Section  \ref{sec:2},
    the coefficients $a_{\mu}$ and  $b_{\mu}$ can be
    expressed
    in function of the coefficients of the surface
    presented
     in a Monge chart, as was done in  Proposition \ref{prop.axial},
considering the coefficients in function of the parameter $\mu\in{\Vv}$.

Define ${\Ff}(\mu)=   \Delta(a(\mu),b(\mu))$ whose zeros define locally the manifold of immersions with
an
  $E_{34}^1$ axiumbilic point.
Here, $\Delta(a,b)$,
given by equation \eqref{eq:delta},
 is the discriminant of the polynomial $R(p)= (p^4-6p^2+1)+(1-p^2)(a+bp)$.
 
Notice that
 due to the particular representation of the 3-jets
taken here, the condition $a(\mu)=-1$   in
Definition \ref{defn.e34},
the jet extension of the immersion
 is not transversal,
but tangent, to the manifold of jets with $E_{34}^1$ axiumbilic points.
 It is always  possible, by an appropriate rotation in the plane $\{x,y\}$ to suppose that $a(\alpha)\notin \{-\frac{27}{2},-1\}$.
 See Section  \ref{sec:2}.

 Assertions $(ii),(iii)$ and $(iv)$ follow from
  the
 definition of ${\Ff}$ and
 the
 previous analysis on the sign of
 the
 discriminant  $\Delta(a_{\mu},b_{\mu})$.

Moreover, the derivative of ${\Ff}(\mu)$ in the
direction of the
coordinate $a$ does not
vanish, leading to conclude that
 $d {\Ff}_{\alpha} \neq 0$.

In fact, assuming $s_{11}(\alpha)=\frac 12 r \ne 0$, it follows that
$$\aligned a_0(\mu)=&y+0(2),\\
a_1(\mu)(x,y)= &-\frac{4(r(\mu)^2+s(\mu)^2)\alpha_2(\mu)}{r(\mu)\left(  r(\mu)\alpha_3(\mu)+s(\mu)\alpha_4(\mu)\right)}x\\
+&\frac{4(s(\mu)\alpha_3(\mu)-r(\mu)\alpha_4(\mu))}{r(\mu)\alpha_3(\mu)+s(\mu)\alpha_4(\mu)}y +O(2)\\
=&a(\mu)x+b(\mu)y+O(2),\\
 \alpha_1=&s_{12} -s_{30}+2 r_{21},\; \alpha_2= r_{30}-r_{12}+2s_{21},\\
 \;  \alpha_3=& s_{03}-s_{21}+2r_{12},\;  \alpha_4= r_{21}-r_{03}+2s_{12}
\endaligned
$$

Consider the deformation
$$\mu= \left( x,y,R_{\alpha}(x,y),S_{\alpha}(x,y)\right)+\left(0,0,t(\frac 16 x^3-\frac 12 xy^2), t x^2y \right).$$
Then, as $\alpha_2=r_{30}-r_{12}+2s_{21}$, it follows that $a(\mu)=-\frac{4(r ^2+s ^2)(\alpha_2+t)}{r \left(  r \alpha_3 +s \alpha_4 \right)}$ and

$$\frac{d}{dt}\left(\Delta(a(\mu),b(\mu)\right)\bigg |_{t=0}=\frac{\partial\Delta}{\partial a}\cv \frac{da}{dt}=\frac{\partial\Delta}{\partial a}\cv \left(-\frac{4(r^2+s^2)}{r(r\alpha_3+s\alpha_4)}\right) \ne 0.$$

In the case where $s_{11}(\alpha)=0$
 it follows that $r_{11}(\alpha)=-\frac 12 s\ne 0$, $\alpha_1\alpha_4\ne 0$ and $\alpha_2(\mu)=0$.
Now consider the  deformation
$$\mu= \left( x,y,R_{\alpha}(x,y),S_{\alpha}(x,y)\right)+\left(0,0,t x^2y,t(-\frac 16 x^3+\frac 12 xy^2 ) \right).$$

Then, $a(\mu)=\frac{ 4(\alpha_1+t)}{\alpha_4 }$ and

$$\frac{d}{dt}\left(\Delta(a(\mu),b(\mu)\right) \bigg |_{t=0}=\frac{\partial\Delta}{\partial a}\cv \frac{da}{dt}=\frac{\partial\Delta}{\partial a}\cv \left(\frac{ 4}{\alpha_4 }\right) \ne 0.$$

\end{proof}

\begin{figure}[h]
       \centering  %
       \fbox{\includegraphics[scale=0.29]{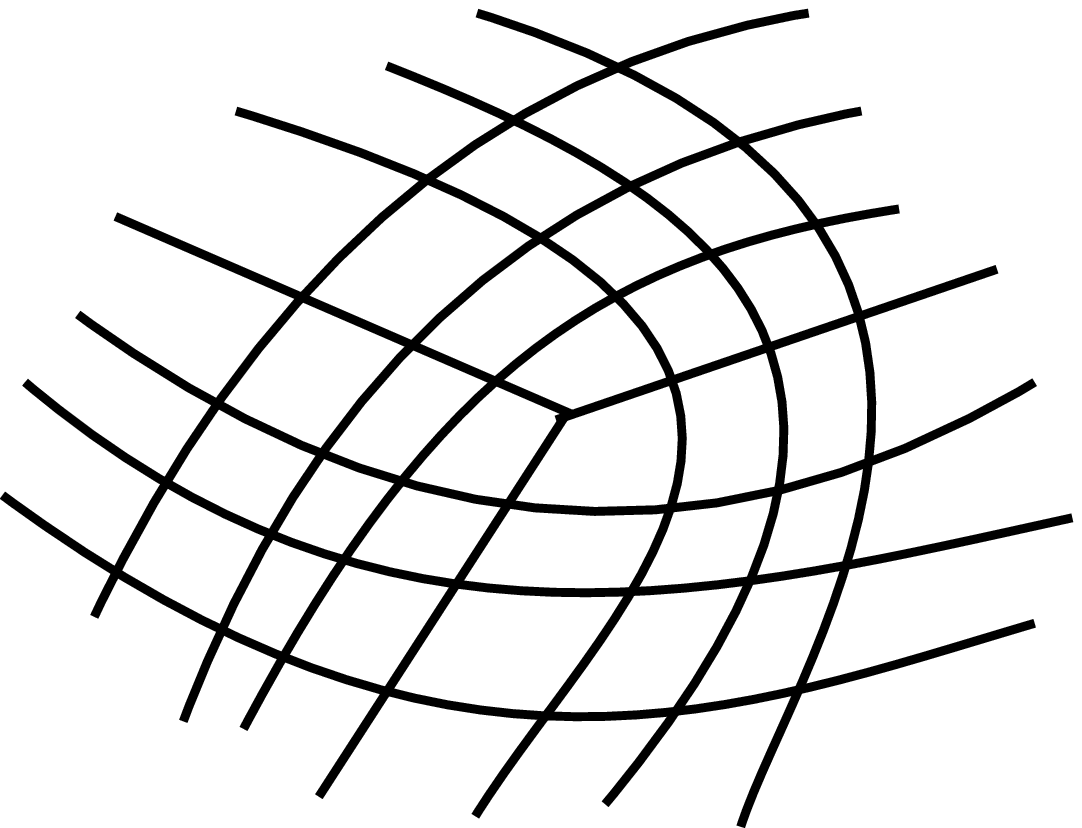}\;\;\;\;\;\;\;\;\;\;\;\;
       			 \includegraphics[scale=0.45]{e34.eps}\;\;\;\;\;\;\;\;\;\;\;\;
       			 \includegraphics[scale=0.29]{e41.eps}}
			  \fbox{\includegraphics[scale=0.28]{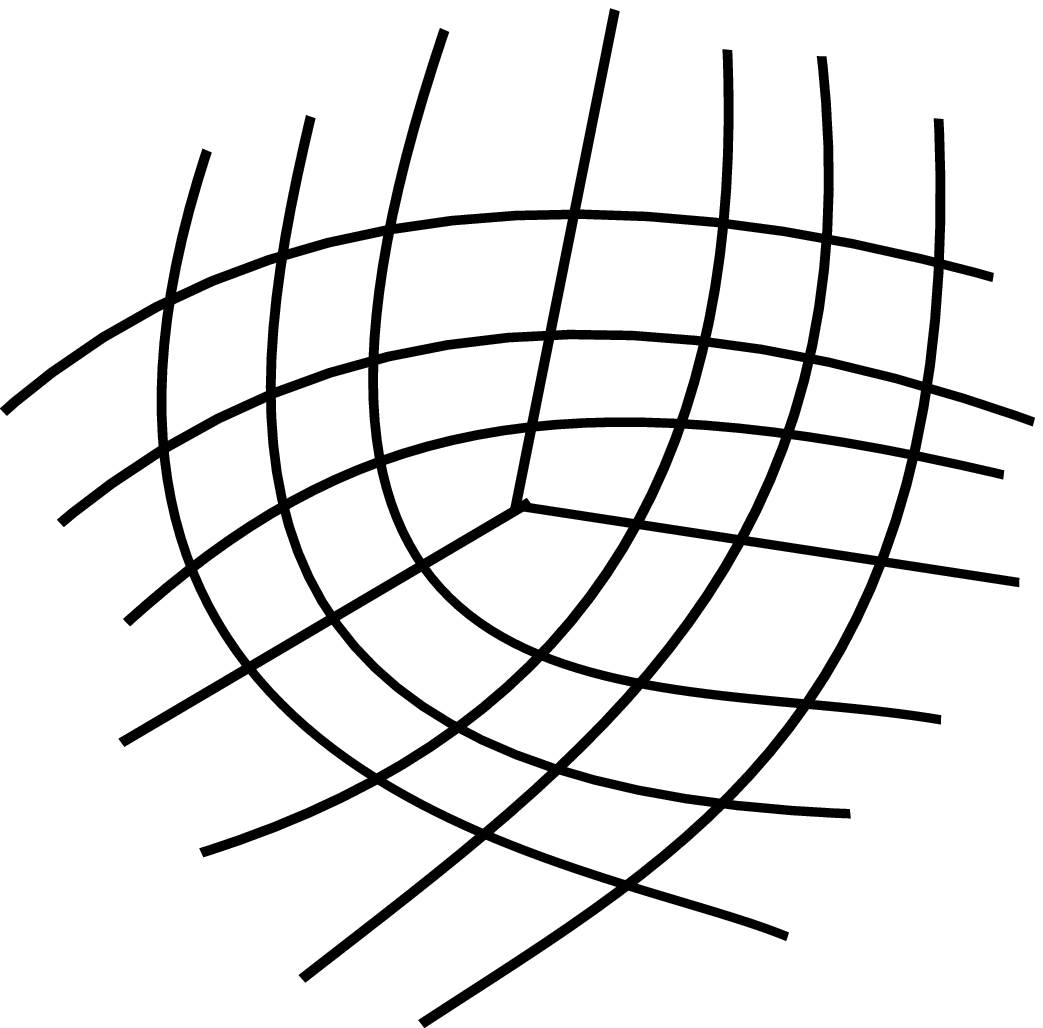}\;\;\;\;\;\;\;\;\;\;\;\;\;\;
       			 \includegraphics[scale=0.45]{e342.eps}\;\;\;\;\;\;\;\;\;\;\;\;\;\;\;
       			 \includegraphics[scale=0.28]{e42.eps}}
\caption{Axial configuration near axiumbilic points.    $E_3$ (left), $E^1_{34}$ (center)   and $E_4$ (right).}
\end{figure}
\begin{figure}[h]
       \centering  
       \fbox{\includegraphics[scale=1.0]{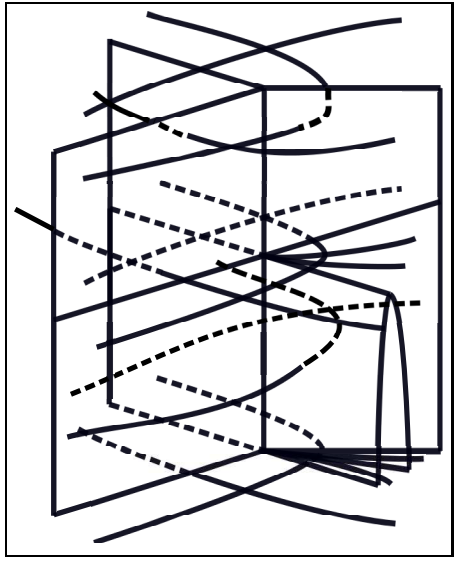}}
\caption{ Bifurcation diagram of the axial configuration near  an axiumbilic point $E^1_{34}$  and the structure of separatrices.}
\label{bifu.e34}
\end{figure}

\subsection{The axiumbilic point $E^1_{4,5}$}

Consider the Monge chart described by equations
  (\ref{Rx}) and (\ref{S}).
  Suppose that the origin is an axiumbilic point, which is expressed by
\begin{equation}
\label{Rmonge}
\aligned
R(x,y)=& \frac{r_{20}}{2}x^2+r_{11}xy+\frac{r_{02}}{2}y^2+\frac{r_{30}}{6}x^3+\frac{r_{21}}{2}x^2y+\frac{r_{12}}{2}xy^2+\frac{r_{03}}{6}y^3\\
+ &\frac{r_{40}}{24}x^4+\frac{r_{31}}{6}x^3 y+\frac{r_{22}}{4}x^2 y^2+ \frac{r_{13}}{6}xy^3+\frac{r_{04}}{24}y^4+h.o.t.,\endaligned
\end{equation}
\begin{equation}
\label{Smonge}\aligned
S(x,y)= &\frac{s_{20}}{2}x^2+s_{11}xy+\frac{s_{02}}{2}y^2+ \frac{s_{30}}{6}x^3+\frac{s_{21}}{2}x^2y+\frac{s_{12}}{2}xy^2+\frac{s_{03}}{6}y^3\\
+ &\frac{s_{40}}{24}x^4+\frac{s_{31}}{6}x^3 y+\frac{s_{22}}{4}x^2 y^2+ \frac{s_{13}}{6}xy^3+\frac{s_{04}}{24}y^4+h.o.t.,\endaligned
\end{equation}
where, $r_{02}=r_{20}+r,\;  r_{11}=-\frac 12 s,\; s_{02}=s_{20}+s,\;  s_{11}= \frac 12 r. $

\nin Let
$\alpha_1=s_{12} -s_{30}+2 r_{21},\; \alpha_2= r_{30}-r_{12}+2s_{21},\;  \alpha_3= s_{03}-s_{21}+2r_{12},\;  \alpha_4= r_{21}-r_{03}+2s_{12},$
$\beta_1= s_{22}-s_{40}+2r_{31},\; \beta_2= r_{40}-r_{22}+2s_{31}, \; \beta_{3}= s_{13}-s_{31}+2r_{22}, \; \beta_{4}= r_{31}-r_{13}+2s_{22},$
\newline
$\; \beta_{5}=s_{04}-s_{22}+2r_{13},\;
\beta_{6}= r_{22}-r_{04}+2s_{13}.$

The functions $a_0$ and $a_1$ (see Proposition \ref{eq.dif.axial}) are given by
\begin{equation}
\label{a0.xy}
a_0(x,y)=a_{10} x+ a_{01}y+ \frac 12 a_{20}x^2+a_{11}xy+\frac 12a_{02}y^2+h.o.t.
\end{equation}
and
\begin{equation}
\label{a1.xy}
a_1(x,y)=b_{10} x+ b_{01}y+ \frac 12 b_{20}x^2+b_{11}xy+\frac 12b_{02}y^2+h.o.t.,
\end{equation}
where
\begin{equation}\label{eq:coefa}\aligned
a_{10} =& \frac 12(r\alpha_1+ s\alpha_2), \hskip 1cm
a_{01}=  \frac 12(r\alpha_3+ s\alpha_4), \\
a_{20}=& -\alpha_2r_{21}+\alpha_1s_{21}+\bigg [\frac{\beta_1}{4}+\frac{s_{20}}{2} (r_{20}^2+s_{20}^2) \bigg ] r
+\bigg [\frac{\beta_2}{4}-\frac{r_{20}}{2} (r_{20}^2+s_{20}^2)\bigg ]s \\
+&(r_{20}^2-s_{20}^2)sr -\frac 38 (r^2+s^2)(s_{20}r-r_{20}s)+r_{20}s_{20}(s^2- r^2), \\
a_{11}=& -r_{12}\alpha_2 +s_{12}  \alpha_1-r_{21}  \alpha_4+ s_{21}  \alpha_3
-  \bigg [ \frac{\beta_3}{2}+r_{20}(r_{20}^2+s_{20}^2) \bigg ] r\\
+&\bigg [\frac{\beta_4}{2}-s_{20}(r_{20}^2+s_{20}^2) \bigg ]s 
- 2  s_{20}  r_{20}  r  s
 -\frac 12(3s_{20}^2+r_{20}^2) s^2\\
 -&\frac 12(3r_{20}^2+s_{20}^2)r^2
 -\frac 38(r^2+s^2)^2 -\frac 54(r^2+s^2)(r_{20}r+s_{20}s), \\
a_{02}=&-r_{12}  \alpha_4+s_{12}  \alpha_3+ \bigg [\frac{\beta_5}{2}-\frac{s_{20}}{2}(r_{20}^2+s_{20}^2) \bigg ]r
  +\bigg[\frac{\beta_6}{2}+\frac{r_{20}}{2}(r_{20}^2+s_{20}^2) \bigg ] s \\
	+&(-2  s_{20}^2+2  r_{20}^2)  s  r
  +2  s_{20}  r_{20}(  s^2-2  r^2 )+
 -\frac 98(r^2+s^2)(r s_{20}-s r_{20}),\endaligned
\end{equation}

\begin{equation}\label{eq:coefb}\aligned
b_{10} =& 2(s\alpha_1- r\alpha_2),\hskip 1cm
b_{01}=  2(s\alpha_3- r\alpha_4), \\
b_{20}=& \alpha_1^2 + \alpha_2^2 -4(s_{21}\alpha_2 + r_{21} \alpha_1)
 + \bigg [-\beta_2 +2 r_{20}(r_{20}^2+s_{20}^2) \bigg ] r  \\
+& \bigg [\beta_1 +2 s_{20}(r_{20}^2+s_{20}^2) \bigg ] s
- \frac{1}{2}(r^2+s^2)(s_{20}s + r_{20}r) +4(r_{20}s-s_{20}r)^2 ,\\
b_{11}=& 2(\alpha_3 \alpha_1 + \alpha_2 \alpha_4)-4(\alpha_1 r_{12} + \alpha_2 s_{12} + \alpha_3 r_{21} + \alpha_4 s_{21})\\
+& 2 \bigg [-\beta_4 +2 s_{20}(r_{20}^2+s_{20}^2) \bigg ] r 
+2 \bigg [\beta_3 -2 r_{20}(r_{20}^2+s_{20}^2) \bigg ] s
+4(s_{20}^2 - r_{20}^2)rs \\
+& 4 r_{20} s_{20}(r^2 - s^2), \\
b_{02}=& \alpha_3^2 +\alpha_4^2 +4(r_{12}^2+s_{12}^2)+4s_{12}(r_{21}-r_{03})+4r_{12}(s_{03}-s_{21})\\
+ &[-\beta_6 -2 r_{20}(r_{20}^2+s_{20}^2)] r 
+[\beta_5 - 2 s_{20}(r_{20}^2+s_{20}^2)]s
+2(r_{20}^2-3s_{20}^2)s^2 \\
+& 2(s_{20}^2-r_{20}^2)r^2.
\endaligned
\end{equation}

\begin{defn}
An axiumbilic point is said to be of type
  $E^1_{4,5}$ if the variety ${\LL}_{\alpha}$ has exactly 4 singular points which are of Morse type along
  the projective line.
\end{defn}

\begin{prop}
\label{equivalencia}
Consider a Monge chart and a homotety such that the differential equation  of axial lines is written as

$$a_0(x,y)(dx^4-6dx^2dy^2+dy^4)+a_1(x,y)dxdy(dx^2-dy^2)+0(3)=0,$$
where
$$\aligned a_0(x,y)=&   y+\frac 12 a_{20}x^2+a_{11}xy+\frac 12a_{02}y^2+h.o.t., \\
a_1(x,y)=&   b_{01}y+ \frac 12 b_{20}x^2+b_{11}xy+\frac 12 b_{02}y^2+h.o.t.
\endaligned $$

Then the following conditions are equivalent:
\begin{enumerate}[$i)$]

	\item the  curves $a_0=0$ and $a_1=0$ are regular and have  quadratic contact at $0$,
	
	\item the axiumbilic point $0$  is of type   $E^1_{4,5}$,
	
	\item the Lie-Cartan vector field defined in $\LL_{\alpha}$ has a quadratic saddle-node in the projective axis with the center manifold transversal to the projective line.
	
\end{enumerate}
\end{prop}

\begin{proof} The differential equation of axial lines can be written as
$$a_0(x,y)(dx^4-6dx^2dy^2+dy^4)+a_1(x,y)dxdy(dx^2-dy^2)+0(3)=0,$$
where
$$\aligned a_0(x,y)=&a_{10} x+ a_{01}y+\frac 12 a_{20}x^2+a_{11}xy+\frac 12a_{02}y^2+h.o.t. \\
a_1(x,y)=& b_{10} x+ b_{01}y+ \frac 12 b_{20}x^2+b_{11}xy+\frac 12 b_{02}y^2+h.o.t.
\endaligned $$
where the coefficients of $a_0$ and $a_1$ are given by equations \eqref{eq:coefa}
and \eqref{eq:coefb}. Here $O(3)$ means terms of order greater than  or equal to $3$ in the variables $x$ and $y$.

In
 what follows
it will be considered a Monge chart such that $a_{10}=0$. This is possible as shown in lemma \ref{lema:rota} and Proposition \ref{prop.axial}. Since the contact between $a_0=0$ and $a_1=0$ is supposed to be quadratic it
results that
$b_{10}=0$ and $a_{01}\cv b_{01}\ne 0$. Also by a homotety it is possible to obtain
$a_{01}=1$.

So, it results that:
\begin{eqnarray}
\label{formaa0}
a_0(x,y) &=&  y+\frac 12 a_{20}x^2+a_{11}xy+\frac 12a_{02}y^2+h.o.t. \\
a_1(x,y) &=&  b_{01}y+ \frac 12 b_{20}x^2+b_{11}xy+\frac 12 b_{02}y^2+h.o.t.
\end{eqnarray}

  Therefore, the  condition of quadratic contact between the two regular curves is expressed by
  $\chi= b_{20} -a_{20}b_{01}\ne 0$.

  \begin{af}
	\label{saddlenode}
In the neighborhood of  $(0,0,0)$, the  Lie-Cartan vector field restricted to the surface  ${\mathcal G}=0$, can be expressed in the chart    $(x,p)$ by
\begin{equation}
\label{campo.cartaxp}
\left \{
\begin{array}{l}
\dot{x}= \frac{\chi}{2}x^2+O(3),\\
\dot{p}= -p+\frac 32 a_{11}a_{20}x^2-
(a_{11}+\chi)p-b_{01}p^2+0(3)
\end{array}
\right .
\end{equation}
and $(0,0,0)$ is a saddle-node when $\chi\ne 0$.
\end{af}

\begin{prova}
Since ${\Gg}_y(0,0,0)=1 \neq 0$, it follows from Implicit Function Theorem that locally   $y=y(x,p)$ and  ${\Gg}(x,y(x,p),p)=0$.

The Taylor expansion of  $y(x,p)$ in the neighborhood of  $(x,p)=(0,0)$ is given by:
\begin{equation}
\label{yxp}
y(x,p)= -\frac 12 a_{20}x^2+O(3).
\end{equation}
The Lie-Cartan vector field restricted to the surface ${\mathcal G}=0$ is given by

\begin{equation*}
\label{campo.carta.xp}
\left \{
\begin{array}{l}
\dot{x}={\Gg}_p(x,y(x,p),p) =\frac 12 \chi x^2+O(3)\\
\dot{p}= -({\Gg}_x +p{\Gg}_y)(x,y(x,p),p)=-p+\frac 32 a_{11}a_{20}x^2-
(\chi + a_{11})p-b_{01}p^2+0(3)
\end{array}
\right .
\end{equation*}

The eigenvalues of the vector field  \eqref{campo.cartaxp} at $0$ are  $\lambda_1=0$  and $\lambda_2=-1$ with respective associated eigenspaces   $\ell_1=(1,-a_{20})$ and $\ell_2=(0,1)$.
By Invariant Manifold Theory the center manifold is tangent to $\ell_1$ and is given by
$W^c=\{(x,-a_{20}x+\frac 32 a_{20}(\chi+a_{11})x^2+O(3))\}$.

The restriction  of the vector field  \eqref{campo.cartaxp}  to the center manifold is given by
$[\frac 12 \chi x^2+0(3)]\frac{\partial }{\partial x}$.
 \end{prova}

\begin{af}
\label{4.pontos.criticos}
The function ${\mathcal G}$ has    exactly  $4$  critical points in the projective line, and they are of Morse-type of index 1 or 2 if and only if  $\chi \neq 0$.
\end{af}
\begin{proof}
\nin   The critical points of $\mathcal G$ along the projective line are determined by
\begin{equation}
\label{criticos}
S(p)=  {\mathcal G}_v(0,0,p)=(p^4-6p^2+1)+b_{01}p(1-p^2)=0,
\end{equation}
which has for  $4$ simple real roots located in the intervals $(-\infty,-1), (-1,0), (0,1)$ and $(1,\infty)$.   This follows from      $S(\pm 1)=-4 $,  $S(0)=1$ and from the discriminant
   $\Delta(S)=4(16 +b_{01}^2)^3>0$.

\nin Along the projective line, the determinant of the Hessian of $\mathcal G$  is given by

\begin{equation}
\label{hessiana}
\Hess {\Gg}(0,0,p)=-(a_{20}(1-6p^2+p^4)+b_{20}p(1-p^2))(b_{01}-12p-3b_{01}p^2+4p^3)^2 .
\end{equation}

The resultant between $S(p)$ and $\Hess {\Gg}(0,0,p)$ is given by
$256 \chi^4(16+b_{01}^2)^6$ and therefore $\Hess {\Gg}(0,0,p)\ne 0$ at the critical points of
$\Gg$. This implies that the critical points are of Morse type.
As $\Gg(0,0,p)=0$ it follows that the index of a critical point is $1$ or $2$ and so locally the level set   $\Gg=0$ is a cone.

The eigenvalues of the derivative of the Lie-Cartan vector field at a point $(0,0,p)$ are
given by:

$\lambda_1= -p(-4p^3+3b_{01}p^2+12p-b_{01}),\;\;\; \lambda_2=-1+18p^2-5p^4-2b_{01}p+4b_{01}p^3.$

At the critical points $p_i$ (satisfying $S(p_i)=0$) it follows that $\lambda_1=-\lambda_2=\frac{p^6+3p^4+3p^2+1}{p^2-1}$, then $\lambda^i_1 \lambda^i_2 <0$, for $i=1..4$.
 
Therefore, these 4 points are saddles of the Lie-Cartan vector field.
As the projective line is invariant it is follows that the other invariant manifold (stable or stable) of a singular point is transversal to the projective line.  \end{proof}

\end{proof}

\begin{prop}
Let $\alpha \in {\Ii}^r$, $r\geq 5$ and $\p$ be an axiumbilic point.
Suppose, in the Monge chart expressed by equations  (\ref{Rmonge}) and (\ref{Smonge}), that  $\alpha_1=\alpha_3=0$ and $\chi \ne 0$.  \; Then $\p$  is an axiumbilic point of type   $E^1_{4,5}$ and the axial configurations of  $\alpha$ in a neighborhood of  $\p$ is as shown in Figure  \ref{desenho.e45}.
\end{prop}

\begin{figure}[h]
       \centering
        \fbox{
       \includegraphics[scale=0.7]{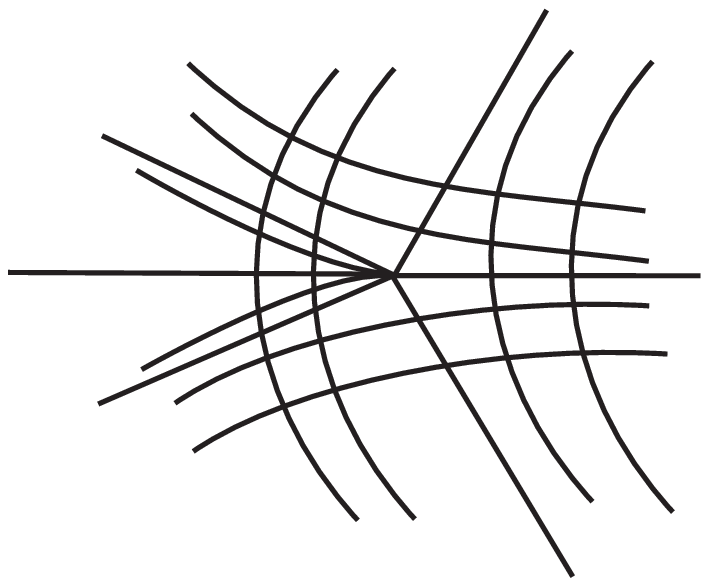}
       \includegraphics[scale=0.8]{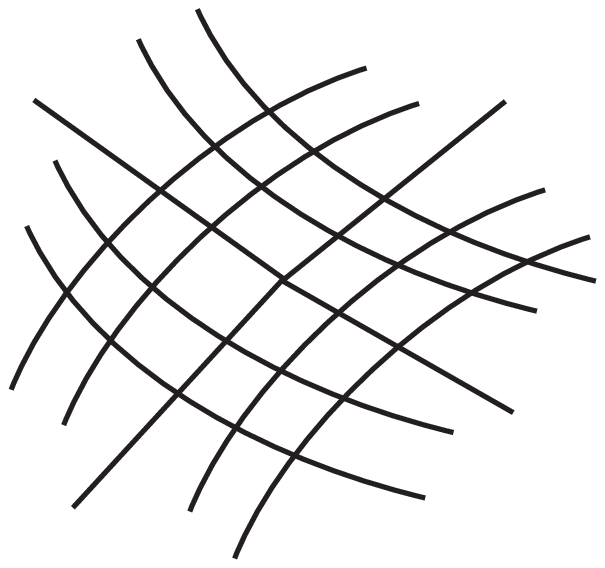}}
\caption{Axial configurations in a neighborhood of an axiumbilic point of type $E^1_{4,5}.$}
\label{desenho.e45}
\end{figure}

\begin{figure}[h]
       \centering  
       \fbox{\includegraphics[scale=0.3]{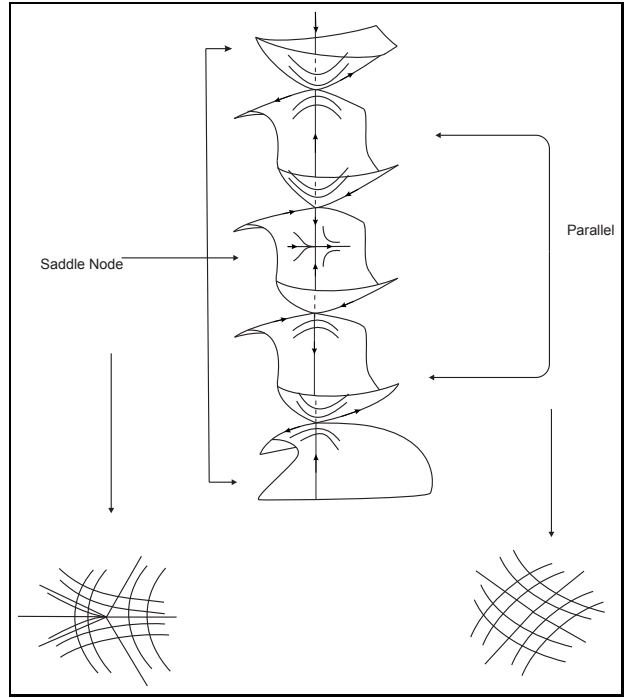}}
\caption{  Lie-Cartan vector field near an axiumbilic point  $E^1_{45}$ and the axial configuration (principal and mean).}
\label{linhas.superficie.folhas}
\end{figure}
\begin{proof}
\nin Condition $\alpha_1=\alpha_3=0$ implies the non-transversal contact of the curves $a_0=0$ and $a_1=0$ at the axiumbilic point $\p$ expressed in the Monge chart by $(0,0)$. By Lemma \ref{lema:rota} and Proposition \ref{prop.axial}, it is possible to express these curves as in equation (\ref{formaa0}). Assuming $\chi \neq 0$, we have the quadratic contact of the curves at the axiumbilic point.

\nin Proposition \ref{equivalencia}  implies that over the axiumbilic point we have five equilibria of the Lie-Cartan vector field. One of them is a regular point of the Lie-Cartan surface, and this is an equilibrium of saddle-node type with center manifold transversal to the axis $p$ (see Claim \ref{saddlenode}).

\nin The remaining equilibria are critical points of Morse type of the Lie-Cartan surface. In the neighborhood of these points, the level set ${\Gg}=0$ are locally cones, and the $4$ points are saddles of the Lie-Cartan vector field (see Claim \ref{4.pontos.criticos}).

\nin Therefore, we conclude the configuration described in Figure \ref{linhas.superficie.folhas}, whose projection of the saddle-node and parallel sectors describe the principal axial and mean axial configurations close to the axiumbilic point $\p$ of type $E^1_{45}$ (Figure \ref{desenho.e45}).
\end{proof}

\begin{prop}
\label{proposicao.e45}
Let $\alpha \in {\Ii}^r$, $r\geq 5$, be an immersion having an axiumbilic point $\p$.
 Then, there exist a neighborhood  $V$ of   $\p$, a neighborhood ${\Vv}$ of ${\alpha}$
 and a function
$F:{\Vv}\fun \R$
of class  ${\Cc}^{r-3}$ such that:
\begin{enumerate}[$i)$]
	\item $dF_{\alpha}\neq 0$,
	\item $F(\mu)=0$ if, and only if, $\mu\in {\Vv}$ has just one axiumbilic point in  $V$, which is of type $E^1_{4,5}$,
	\item $F(\mu)<0$ if, and only if,  $\mu$ has  exactly two axiumbilic points
	in   $V$, one of type $E_4$ and the other of type $E_5$,
	\item $F(\mu)>0$ if, and only if,  $\mu$ has no axiumbilic points in $V$.
\end{enumerate}
\end{prop}
\begin{proof}
By Proposition \ref{equivalencia},   $\alpha$ being an immersion having an axiumbilic point $\p$  of type  $E^1_{45}$, the curves $a_0^{\alpha} = 0$ and
$a_1^{\alpha}= 0$ have quadratic contact at $\p$.

\nin Since $\frac{\partial a_0^{\alpha}}{\partial y}(0,0) = a_{01} \neq 0$,
if follows from Implicit Function Theorem
that locally, for    $\mu$ in a neighborhood  ${\Vv}$  of $\alpha$, $y=y_{\mu}(x)$ and $a_0^{\mu}(x,y_{\mu}(x))=0$.

\nin Moreover, $\frac{\partial^2 a_1^{\alpha}}{\partial x^2}(0,0) = b_{20} \neq 0$, and so   $x=x_{\mu}$ is a local solution of
$\frac{\partial a_1^{\mu}}{\partial x}(x_{\mu},y_{\mu}(x_{\mu}))=0$.

\nin Define
${\Ff}(\mu)=a_1^{\mu}(x_{\mu},y_{\mu}(x_{\mu}))$.
Consider the variation $h_t(x,y)=(x,y,R(x,y)+txy,S(x,y)+txy)$. It follows that $\frac{dF(t)}{dt}\bigg |_{t=0}\neq 0$, and so
$dF_{\alpha}\neq 0$. Therefore, the result follows from the Implicit Function Theorem.

The axiumbilic point of type  $E^1_{45}$ is therefore the transition between zero and two axiumbilic points, one of type $E_4$ and the other of type $E_5$.

 In Figures \ref{transition} and \ref{fig:sup45} are  illustrated this transition, with the axial configurations  sketched  in  two different styles. See also Figure  \ref{fig:sup45} for an illustration of transition in the Lie - Cartan surface. 
 
 \end{proof}
 
\begin{figure}[h]
       \centering  
       \fbox{\includegraphics[scale=0.5]{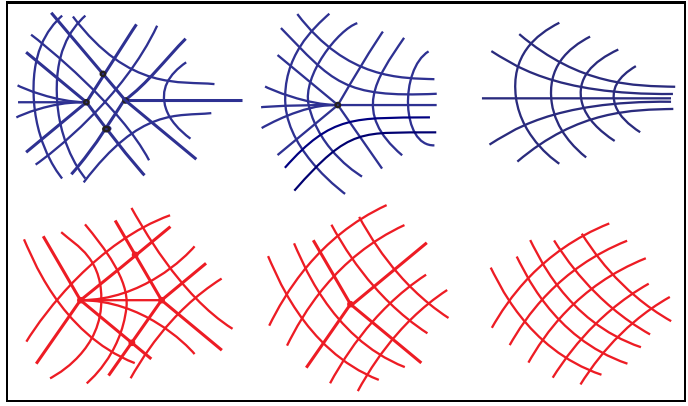}}
\caption{Axiumbilic point $E^1_{45}$. The axiumbilic points $E_4$ and $E_5$ collapse  in
an axiumbilic point  $E^1_{45}$, and after they  are eliminated and there are no axiumbilic points. }
\label{transition}
\end{figure}

\begin{figure}[ht]
       \centering  
       \fbox{\includegraphics[scale=0.8]{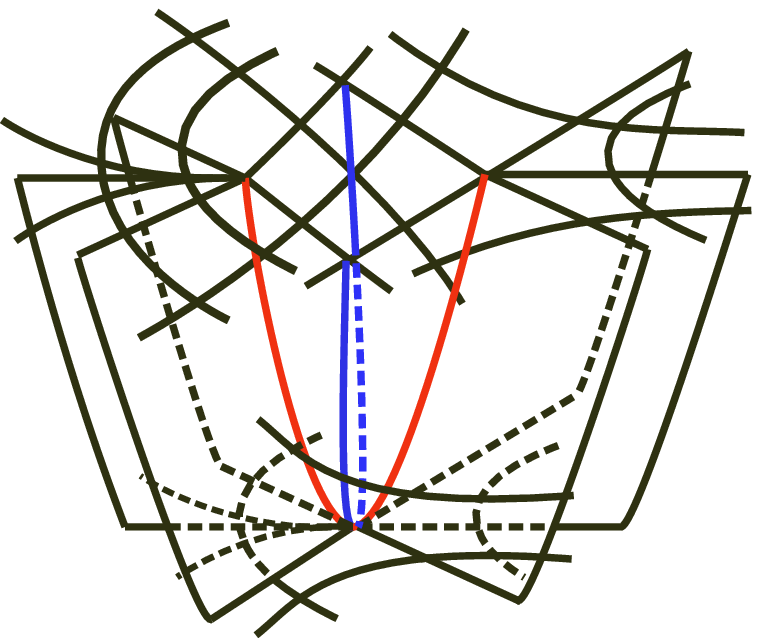}}
\caption{Bifurcation diagram of the axial configuration  near an axiumbilic  point of type    $E^1_{45}$ and the structure of separatrices}
\label{bifu.e45}
\end{figure}

\begin{figure}[h]
       \centering  
       \fbox{\includegraphics[scale=0.56]{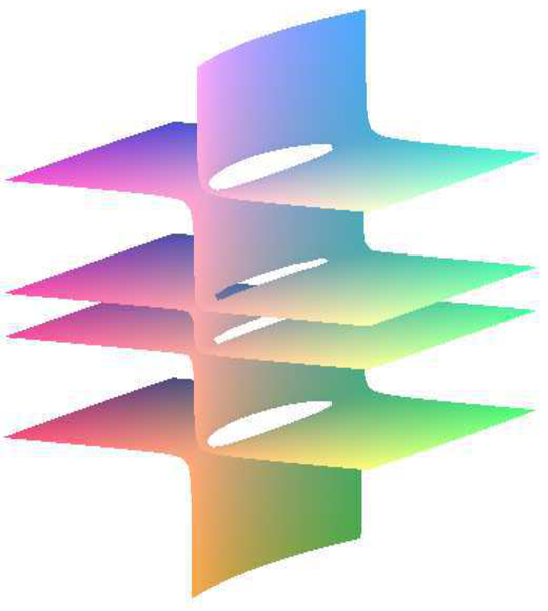}
             \includegraphics[scale=0.56]{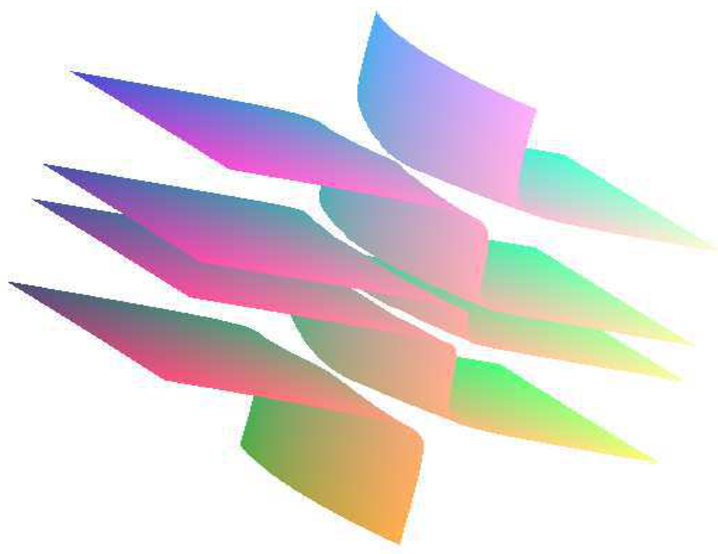}
             \includegraphics[scale=0.56]{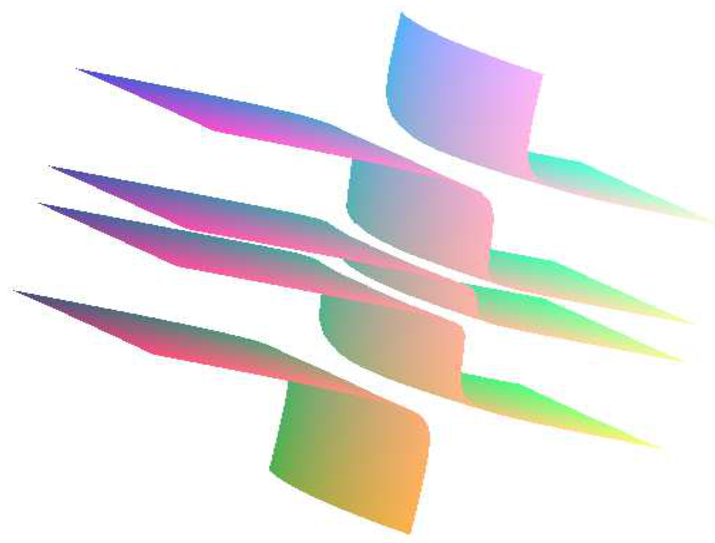}}
\caption{The  Lie-Cartan surface. In the left, with two axiumbilic point, in the center with four singular points, and in the right the four regular levels.}
\label{fig:sup45}
\end{figure}

\begin{prop}
\label{teo.transversality}
In the space of smooth mappings of $M  \times \R \fun \R^4$ which are immersions relative to the first variable, those which have  all their axiumbilic points either  generic (of types $E_3$, $E_4$ and $E_5$)
or, transversally, of  types  $E^1_{34}$ and  $E^1_{45}$ is open and dense.
Furthermore, for such families the axiumbilic points  describe  a regular curve in $M  \times \R$ whose projection into $\R$ has only non-degenerate critical points at $E^1_{45}$ and the regular points
of the projection is a collection of arcs bounded by $E^1_{34}$ points, which a the common boundary of
$E_{3}$ and $E_{4}$ arcs.
\end{prop}

Proposition \ref{teo.transversality}
 follow from the analyses  in propositions
 \ref{proposicao.e34} and \ref{proposicao.e45} and an application of Thom Transversality Theorem to the submanifold of four jets of immersions at axiumbilic points, stratified by the generic axiumbilic points, by $E^1_{34}$ and type $E^1_{45}$, and their complement. See Section \ref{sec:4}.

 \section{Transversality and Stratification}
 \label{sec:4}
 
 Consider the  space $\mathbb{J}^{k}(M,\mathbb{R}^4)$  of $k$-jets  of  immersions $\alpha$ of a compact oriented surface $M $  into $\mathbb{R} ^4$, endowed with the  structure of Principal Fiber Bundle. The base is  $M$; the  fiber is the space $\mathbb{R} ^4\times \mathbb{J}^{k}(2,4)$, where $\mathcal{J}^{k}(2,4)$ is the space of  $k$-jets of immersions  of $\mathbb{R}^2$ to $\mathbb{R} ^4$,  preserving the respective  origins. The structure group,  $\mathbb{A}_{+}^{k}$,  is  the product of the group  of $\mathcal{L}_{+}^{k}(2,2)$ of  $k$-jets of origin and  orientation preserving diffeomorphisms of $\mathbb{R}^2$, acting on the right by coordinate changes,  and the group $\mathbb{R} ^4\times \mathcal{O}_{+}(4,4)$ of positive isometries, acting on the left, consisting on a translation, taken as a vector  in the first factor,  and a positive rotation of $\mathbb{R}^4$, taken  on the second factor.
 Denote by $\Pi_{k,l}, k\leq l$ the projection of $\mathcal{J}^{l}(2,4)$ to $\mathcal{J}^{k}(2,4)$. It is well known that the group action commutes with projections.

\begin{defn} \label{def:cs}
We define below  the {\it canonic  axiumbilic stratification} of
$\mathcal{J}^{4}(2,4)$. The term {\it  canonic} means that the strata are invariant under
 the action of the group  $\mathbb{A}_{+}^{k}$= $\mathcal{O}_{+}(4,4)\times\mathcal{L}_{+}^{k}(2,2)$.

 \begin{itemize}
\item[1)]
 {\it aximbilic Jets}:   ${\mathcal U}^{4}$, those  in the orbit of  $j^{4}(x,y,R(x,y),S(x,y))$, where   $R$ and $S$ are as  in equations \eqref{Rx} and \eqref{S} satisfying the axiumbilic conditions defined in terms of $j^2R(0)$ and $j^2S(0).$ It is a closed variety of codimension $2$.

\item[2)] {\it Non-axiumbilic Jets}: $(\mathcal{NU})^{4}$ is the
complement  of ${\mathcal U}^{4}$. It is an open submanifold of
codimension $0$.

\item[3)] {\it Non-stable axumbilic Jets}: $(\mathcal{NE})^{4}$,
in the orbit of the axiumbilic jets for which:

\noindent $\bullet$\;\; $T=(\alpha_1\alpha_4-\alpha_2\alpha_3)(r^2+s^2)= 0$
or

\noindent $\bullet$\;\;
  $T\neq 0$ and conditions that characterize $E_3$ or $E_4$ axiumbilic points in Proposition \ref{axi.estaveis} fail.
\end {itemize}
It is a closed variety of  codimension 3, which can be expressed as the union of the  following invariant strata:
\begin {itemize}

 \item[3.1)] {\it Non-Transversal jets}:   $\mathcal{E}^1_{45}$ for which  $T= 0$ and $\chi\ne 0$.  It  has codimension 3.

\item[3.2)] {\it Transversal-double jets}:
$(\mathcal{E}^1_{34})^{4}$, The Lie-Cartan   field  has a quadratic saddle-node in the projective line which is characterized by Proposition \ref{proposicao.e34}. Has codimension 3.

\end {itemize}

\begin{itemize}
\item[4)] The {\it stable axumbilic jets}: $\mathcal{UE}^4$, the
complement in $\mathcal{U}^{4}$ of   $\mathcal{NE}^{4}$.
 \end {itemize}
\end {defn}

\begin{prop} \label{th:trans}
In the space of 1-parameter families of immersions, those whose 4-jet extension are transversal to  the canonical axiumbilic stratification is open and dense.
\end{prop}

\begin {proof}
Follows from Thom Transversality Theorem \cite {lt}.
\end {proof}

\section{Concluding Comments}
\label{sec:5}

In this work
was
 established the principal axial  and the mean axial configurations in a  neighborhood of the axiumbilc points
of  types
 $E^1_{34}$ and  $E^1_{45}$.
 The approach concerning  methods and class of differentiability requirements is  distinct
 from
  that
  presented in the work of  Guti\'errez-Gu\'inez-Casta\~neda in \cite{gutierrez3}.
  The use of the  Lie-Cartan suspension method made possible the  study of these points by means the  classic theory of  differential equations,
  in clear analogy
  with
  the saddle-node  bifurcation of vector fields in the plane, following \cite{andronov},  \cite{soto1} and  \cite{soto-ga-gu}.

  The type $E^1_{34}$ satisfies the  transversality condition  of the curves $a_0$ and $a_1 $, Proposition \ref{prop.axial},
  which amounts to the fact the Lie-Cartan
  surface
  remains  regular in a neighborhood of the projective axis  over the axiumbilic point.
  In this case  there  is a saddle-node   equilibrium point of  the  Lie-Cartan vector field whose central
  separatrix is  along the   projective axis itself.
   The  axial configurations are established in Proposition \ref{proposicao.e3434}  and the qualitative change (bifurcation) between the types  $E_3$ and  $E_4$, with the  variation of
   a
	one parameter in the  space of   immersions, is explained in Proposition \ref{e34}. See Figure \ref{bifu.e34}.

In the case $E^1_{45}$   the transversality  condition  fails, since  curves  $a_0$ and  $a_1$, Proposition \ref{equivalencia},   have
quadratic contact at the  axiumbilic point.
  Here the Lie-Cartan  surface is  not regular along the  projective axis.
  It is established in   Proposition  \ref{equivalencia}  that there are four
   conic critical points of  Morse type on the  $p-$axis.
   At these points there are partially hyperbolic equilibria of the Lie-Cartan vector field.
   There is also a  saddle-node  equilibrium in the regular part of
    the surface whose central  separatrix  is  transversal to the  projective axis.
     The integral curves of the Lie - Cartan vector field
     on the regular  components  of the Lie - Cartan (which are four bi-punctured disks)  are   illustrated
     in Figure \ref{linhas.superficie.folhas}.
     Their projections on the plane give the axial configurations in  a neighborhood of the axiumbilic point.

     In Proposition \ref{teo.transversality} is established the one parameter variation (bifurcation) in the space of immersions. This leads to
      the fact that for small   perturbations of an immersion  with an  axiumbilic point of this type it holds that  two  axiumbilic  points, one of type  $E_4$ and the other of type  $E_5$, bifurcate form $E^1_{45}$  or disappear leaving a neighborhood free from  axiumbilic points, in full analogy with the saddle-node bifurcation \cite{andronov} and \cite{soto1}. See Figure \ref{bifu.e45}.

  In  Theorem \ref{th:trans} the genericity of the points $E^1_{34}$ and $E_{45}^1$ is established in terms of stratification and transversality.

\vskip 1cm

\author{\noindent Ronaldo Garcia\\Instituto de Matem\'atica e Estat\'{\i}stica \\
Universidade Federal de Goi\'as,\\ CEP 74001--970, Caixa Postal 131 \\
Goi\^ania, Goi\'as, Brazil \\
 \email{ragarcia@mat.ufg.br  }
\vskip 0.5cm

\vskip 0.5cm
\author{\noindent Jorge Sotomayor\\ Instituto de Matem\'atica e Estat\'{\i}stica \\
Universidade  de S\~ao Paulo,\\
 Rua do Mat\~ao  1010,
Cidade Univerit\'aria, CEP 05508-090,\\
S\~ao Paulo, S. P, Brazil \\
 \email{sotp@ime.usp.br}
 \vskip 0.5cm

 \vskip 0.5cm
\author{\noindent Flausino L. Spindola\\ Instituto de Matem\'atica e Estat\'{\i}stica \\
Universidade  de S\~ao Paulo,\\
 Rua do Mat\~ao  1010,
Cidade Univerit\'aria, CEP 05508-090,\\
S\~ao Paulo, S. P, Brazil \\
 \email{flausino@ime.usp.br}

\end{document}